\documentclass{svjour3}                     
\smartqed  
\usepackage{graphicx}
\usepackage{amsmath}       
\usepackage{amssymb}
\usepackage{mathptmx}      
\usepackage{multicol}
\usepackage{ifpdf}
\usepackage{hyperref}

\newtheorem{algorithm}{\normalfont\textsc{Algorithm}}
\newtheorem{assumption}{A.\!}

\newcommand{\norm}[1]{\|#1\|}
\newcommand{\abs}[1]{|#1|}
\newcommand{\Eproof}{\hfill$\square$}

\begin{document}

\title{Combining Lagrangian Decomposition and Excessive Gap Smoothing Technique for Solving Large-Scale Separable Convex Optimization Problems}

\titlerunning{Excessive gap smoothing techniques in Lagrangian dual decomposition}

\author{Tran Dinh Quoc \and
        Carlo Savorgnan \and Moritz Diehl 
}


\institute{Tran Dinh Quoc \and Carlo Savorgnan \and Moritz Diehl 
\at {Department of Electrical Engineering (ESAT-SCD) and Optimization in Engineering Center (OPTEC), K.U. Leuven, Kasteelpark Arenberg 10, B-3001 Leuven, Belgium.\\ 
\email{\texttt{\{quoc.trandinh, carlo.savorgnan, moritz.diehl\}@esat.kuleuven.be}}}\\
Tran Dinh Quoc, Hanoi University of Science, Hanoi, Vietnam.
}

\date{Received: date / Accepted: date}

\maketitle

\begin{abstract}
A new algorithm for solving large-scale convex optimization problems with a separable objective function is proposed. The basic idea is to combine three
techniques: Lagrangian dual decomposition, excessive gap and smoothing.
The main advantage of this algorithm is that it dynamically updates the smoothness parameters which leads to numerically robust performance.
The convergence of the algorithm is proved under weak conditions imposed on the original problem. The rate of convergence is $O(\frac{1}{k})$, where $k$ is the iteration counter.
In the second part of the paper, the algorithm is coupled with a dual scheme to construct a switching variant of the dual decomposition. 
We discuss implementation issues and make a theoretical comparison. Numerical examples confirm the theoretical results.

\keywords{Excessive gap \and smoothing technique \and Lagrangian decomposition \and proximal mappings \and large-scale problem \and separable convex optimization \and distributed optimization.}
\end{abstract}

\section{Introduction}
Large-scale convex optimization problems appear in many areas of science such as graph theory, networks, transportation, distributed model predictive control,
distributed estimation and multistage stochastic optimization
\cite{Cohen1978,Holmberg2001,Kojima1993,Komodakis2010,Love1973,Purkayastha2008,Samar2007,Tsiaflakis2010,Vania2009,Venkat2008,Zhao2005}.
Solving large-scale optimization problems is still a challenge in many applications \cite{Conejo2006}.
Over the years, thanks to the development of parallel and distributed computer systems, the chances for solving large-scale problems have been increased.
However, methods and algorithms for solving this type of problems are limited \cite{Bertsekas1989,Conejo2006}.

Convex minimization problems with a separable objective function form a class of problems which is relevant in many applications. This class of problems is also
known as separable convex minimization problems, see, e.g. \cite{Bertsekas1989}.
Without loss of generality, a separable convex optimization problem can be written in the form of a convex program with separable objective function and coupled
linear constraints \cite{Bertsekas1989}. In addition, decoupling convex constraints may also be considered. Mathematically, this problem can be formulated in
the following form:
\begin{equation}\label{eq:separable_convex_problem}
\begin{array}{cl}
\displaystyle\min_{x\in\mathbb{R}^{n}} &\phi(x) := \displaystyle\sum_{i=1}^M\phi_i(x_{i})\\
\textrm{s.t.}~ &x_{i} \in X_{i}~ (i=1,\cdots, M),\\
& \displaystyle\sum_{i=1}^MA_{i}x_{i} = b,
\end{array}
\end{equation}
where $\phi_i:\mathbb{R}^{n_i} \to \mathbb{R}$ is convex, $X_{i} \in\mathbb{R}^{n_i}$ is a nonempty, closed convex set, $A_{i}\in\mathbb{R}^{m\times n_i}$, $b\in\mathbb{R}^m$ for
all $i=1,\dots, M$, and $n_1+n_2 +\cdots + n_M = n$.
The last constraint is called \textit{coupling linear constraint}.
In principle, many convex problems can be written in this separable form by doubling the variables, i.e. introducing new variables $x_{i}$ and imposing the constraint $x_{i} = x$.
Despite the increased number of variables, treating convex problems by doubling variables may be useful in some situations, see, e.g. \cite{Fukushima1996,Goldfarb2010}.
 
In the literature, numerous approaches have been proposed for solving problem \eqref{eq:separable_convex_problem}. For example, (augmented) Lagrangian
relaxation and subgradient methods of multipliers \cite{Bertsekas1989,Hamdi2005,Ruszczynski1995,Vania2009}, Fenchel's dual decomposition \cite{Han1988},
alternating linearization \cite{Boyd2011,Goldfarb2010,Kontogiorgis1996}, proximal point-type methods
\cite{Bertsekas2010,Chen1994,Tseng1997}, interior point methods \cite{Kojima1993,Zhao2005,Mehrotra2009,Tran2011}, mean value cross decomposition
\cite{Holmberg2006} and partial inverse method \cite{Spingarn1985} among many others have been proposed. 
Our motivation in this paper is to develop a numerical algorithm for solving \eqref{eq:separable_convex_problem} which can be implemented in a parallel or
distributed fashion.
Note that the approach presented in the present paper is different from splitting methods and alternating methods considered in the literature, see, e.g.
\cite{Boyd2011,Eckstein1992}. 

One of the classical approaches for solving \eqref{eq:separable_convex_problem} is Lagrangian dual decomposition. The main idea of this approach is to solve the
dual problem by means of a subgradient method. It has been recognized in practice that subgradient methods are usually slow and numerically sensitive to the
step size parameters. In the special case of a strongly convex objective function, the dual function is differentiable. Consequently, gradient schemes can be
applied to solve the dual problem.

Recently, Nesterov \cite{Nesterov2004} developed smoothing techniques for solving nonsmooth convex optimization problems based on the fast gradient scheme which
was introduced in his early work \cite{Nesterov1983}. 
The fast gradient schemes have been used in numerous applications including image processing, compressed sensing, networks and system identification
\cite{Alexandre2008,Bienstock2006,Han-Jakob2005,Hariharan2008,Goldfarb2010,Naveen2007}. 

Exploiting Nesterov's idea in \cite{Nesterov2005a}, Necoara and Suykens \cite{Necoara2008} applied a smoothing technique to the dual problem in the
framework of Lagrangian dual decomposition and then used the fast gradient scheme to maximize the smoothed function of the dual problem. This resulted in a new
variant of dual decomposition algorithms for solving separable convex optimization. The authors proved that the rate of convergence of their algorithm is
$O(\frac{1}{k})$ which is much better than $O(\frac{1}{\sqrt{k}})$ in the subgradient methods of multipliers, where $k$ is the iteration counter. 
A main disadvantage of this scheme is that the smoothness parameter requires to be given \textit{a priori}. Moreover, this parameter crucially depends on the given desired accuracy. Since the Lipschitz constant of the gradient of the objective function in the dual problem is inversely proportional to the smoothness parameter, the algorithm usually generates short steps towards a solution of the problem
although the rate of convergence is $O(\frac{1}{k})$.

To overcome this drawback, in this paper, we propose a new algorithm which combines three techniques: smoothing \cite{Nesterov2005a,Nesterov2005}, excessive gap
\cite{Nesterov2005} and Lagrangian dual decomposition \cite{Bertsekas1989} techniques. 
Instead of fixing the smoothness parameters, we update them dynamically at every iteration.
Even though the worst case complexity is $O(\frac{1}{\varepsilon})$, where $\varepsilon$ is a given tolerance,
the algorithms developed in this paper work better than the one in \cite{Necoara2008} and are more numerically robust in practice. 
Note that the computational cost of the proposed algorithms remains almost the same as in the proximal-center-based decomposition algorithm proposed in
\cite[Algorithm 3.2]{Necoara2008}.
(Algorithm 3.2 in \cite{Necoara2008} requires to compute an additional dual step). 
This algorithm is called \textrm{dual decomposition with primal update} (Algorithm \ref{alg:A1}). 
Alternatively, we apply the switching strategy of \cite{Nesterov2005} to obtain a decomposition algorithm with switching primal-dual update for solving problem \eqref{eq:separable_convex_problem}.
This algorithm differs from the one in \cite{Nesterov2005} at two points. First, the smoothness parameter is dynamically updated with an exact formula and
second the proximal-based mappings are used to handle the nonsmoothness of the objective function. 
The second point is more significant since, in practice, estimating the Lipschitz constants is not an easy task even if the objective function is
differentiable. 
The switching algorithm balances the disadvantage of the decomposition methods using the primal update (Algorithm \ref{alg:A1}) and the dual update (Algorithm
3.2 \cite{Necoara2008}). Proximal-based mapping only plays a role of handling the nonsmoothness of the objective function. Therefore,
the algorithms developed in this paper do not belong to any proximal-point algorithm class considered in the literature.
Note also that all algorithms developed in this paper are first order methods which can be highly distributed.

\vskip0.1cm
\noindent\textbf{Contribution. }
The contribution of this paper is the following:
\begin{enumerate}
\item We apply the Lagrangian relaxation, smoothing and excessive gap techniques to large-scale separable convex optimization problems which are not necessarily smooth. 
Note that the excessive gap condition that we use in this paper is different from the one in \cite{Nesterov2005}, where not only the duality gap is measured but also the feasibility gap is used in the framework of constrained optimization, see \eqref{eq:smooth_f}.
\item We propose two algorithms for solving  general separable convex optimization problems. The first algorithm is new, while the second one is a new variant
of the first algorithm proposed in \cite[Algorithm 1]{Nesterov2005} applied to Lagrangian dual decomposition. A special case of the
algorithms, where the objective is strongly convex is considered. All the algorithms are highly parallelizable and distributed.
\item The convergence of the algorithms is proved and the rate of convergence is estimated. Implementation details are discussed and a 
theoretical and numerical comparison is made. 
\end{enumerate}
The rest of this paper is organized as follows. 
In the next section, we briefly describe the Lagrangian dual decomposition method \cite{Bertsekas1989} for separable convex optimization, the smoothing technique via prox-functions as well as excessive gap techniques \cite{Nesterov2005}. We also provide several technical lemmas which will be used in the sequel. 
Section \ref{sec:decom_primal_acceleration} presents a new algorithm called  \textit{decomposition algorithm with primal update} and estimates its worst-case
complexity. 
Section \ref{sec:decom_primal_dual_acceleration} is a combination of the primal and the dual step update schemes which is called  \textit{decomposition
algorithm with primal-dual update}.
Section \ref{sec:strongly_convex_case} is an application of the dual scheme \eqref{eq:dual_step} to the strongly convex case of problem \eqref{eq:distNLP}.   
We also discuss the implementation issues of the proposed algorithms and a theoretical comparison of Algorithms \ref{alg:A1} and
\ref{alg:A2} in Section \ref{sec:discussion}.
Numerical examples are presented in Section \ref{sec:num_results} to examine the performance of the proposed algorithms and  to compare different methods.

\vskip 0.1cm
\noindent\textbf{Notation. } Throughout the paper, we shall consider the Euclidean space $\mathbb{R}^n$ endowed with an inner product $x^Ty$ for $x,
y\in\mathbb{R}^n$ and the norm $\norm{x} := \sqrt{x^Tx}$. 
Associated with $\norm{\cdot}$, $\norm{\cdot}_{*} := \max\left\{(\cdot)^Tx ~:~ \norm{x}\leq 1\right\}$ defines its dual norm. For simplicity of
discussion, we use the Euclidean norm in the whole paper. Hence, $\norm{\cdot}_{*}$ is equivalent to $\norm{\cdot}$. The notation $x = (x_1, \dots, x_M)$
represents a column vector in $\mathbb{R}^n$, where $x_i$ is a subvector in $\mathbb{R}^{n_i}$, $i=1,\dots, M$ and $n_1+\cdots + n_M = n$.

\section{Lagrangian dual decomposition and excessive gap smoothing technique}\label{sec:decom_smooth}
A classical technique to address coupling constraints in optimization is Lagrangian relaxation \cite{Bertsekas1989}. However, this technique often leads to a nonsmooth optimization problem in the dual form. To overcome this situation, we combine the Lagrangian dual decomposition and smoothing technique in \cite{Nesterov2005a,Nesterov2005} to obtain a smoothly approximate dual problem.

For simplicity of discussion, we consider problem \eqref{eq:separable_convex_problem} with $M=2$. However, the methods presented in the next sections can be
directly applied to the case $M > 2$ (see Section \ref{sec:discussion}). 
The problem \eqref{eq:separable_convex_problem} can be rewritten as follows:
\begin{equation}\label{eq:distNLP}
\phi^{*} := \left\{\begin{array}{cl}
\displaystyle\min_{x:=(x_{1}, x_{2})} & \phi(x):=\phi_1(x_{1}) + \phi_2(x_{2})\\
\textrm{s.t.} &A_{1}x_{1} + A_{2}x_{2} = b\\
&x \in X_{1}\times X_{2} := X,
\end{array}\right.
\end{equation}
where $\phi_i : \mathbb{R}^{n_i}\to\mathbb{R}$ is convex, $X_{i}$ is a nonempty, closed, convex and bounded subset in $\mathbb{R}^{n_i}$, $A_{i}\in\mathbb{R}^{m\times n_i}$ and $b\in\mathbb{R}^m$
($i=1,2$).
Problem \eqref{eq:distNLP} is said to satisfy the Slater constraint qualification condition if $\textrm{ri}(X)\cap \{x =(x_1,x_2) ~|~ A_1x_1 + A_2x_2 = b\} \neq \emptyset$, where
$\textrm{ri}(X)$ is the relative interior of the convex set $X$. 
Let us denote by $X^{*}$ the solution set of this problem.
Throughout the paper, we assume that:
\begin{assumption}\label{as:A1}
The solution set $X^{*}$ is nonempty and either the Slater qualification condition for problem \eqref{eq:distNLP} holds or $X_{i}$ is polyhedral. The function $\phi_i$ is proper, lower semicontinuous and convex in $\mathbb{R}^n$, $i=1,2$. 
\end{assumption}
Since $X$ is convex and bounded, $X^{*}$ is also convex and bounded. 
Note that the objective function $\phi$ is not necessarily smooth. For example, $\phi(x) = \norm{x}_1 = \sum_{i=1}^n|x_{(i)}|$, which is is nonsmooth and separable.

\subsection{Decomposition via Lagrangian relaxation}
Let us define the Lagrange function of the problem \eqref{eq:distNLP} with respect to the coupling constraint $A_{1}x_{1} + A_{2}x_{2} = b$ as:
\begin{equation}\label{eq:Lagrange_func}
L(x, y) := \phi_1(x_{1}) + \phi_2(x_{2}) + y^T(A_{1}x_{1} + A_{2}x_{2} - b),
\end{equation}
where $y \in\mathbb{R}^m$ is the multiplier associated with the coupling constraint $A_{1}x_{1} + A_{2}x_{2} = b$.
A triplet $(x^{*}_{1}, x^{*}_{2}, y^{*}) \in X \times \mathbb{R}^m$ is called a saddle point of $L$ if:
\begin{equation}\label{eq:saddle_point}
L(x^{*}, y) \leq L(x^{*}, y^{*}) \leq L(x, y^{*}), ~ \forall x\in X, ~\forall y\in\mathbb{R}^m.
\end{equation}
Next, we define the Lagrange dual function $d$ of the problem \eqref{eq:distNLP} as:
\begin{equation}\label{eq:dual_func}
d(y) := \min_{x\in X}\left\{L(x, y) := \phi_1(x_{1}) + \phi_2(x_{2}) + y^T(A_{1}x_{1} + A_{2}x_{2} - b) \right\}. 
\end{equation}
and then write down the dual problem of \eqref{eq:distNLP}:
\begin{equation}\label{eq:dual_prob}
d^{*} := \max_{y\in\mathbb{R}^m} d(y). 
\end{equation}
Let $A = [A_{1}, A_{2}]$.
Due to Assumption A.\ref{as:A1} \textit{strong duality} holds and we have:
\begin{equation}\label{eq:minmax_prob}
d^{*} = \max_{y\in\mathbb{R}^m}d(y) \overset{\tiny\textrm{strong~duality}}{=} \min_{x\in X}\left\{ \phi(x) ~|~ Ax = b\right\} = \phi^{*}. 
\end{equation}
Let us denote by $Y^{*}$ the solution set of the dual problem \eqref{eq:dual_prob}. It is well-known that $Y^{*}$ is bounded due to Assumption A.\ref{as:A1}.

Now, let us consider the dual function $d$ defined by \eqref{eq:dual_func}. It is important to note that the dual function $d(y)$ can be computed 
separately as: 
\begin{equation}\label{eq:dual_func2}
d(y) = d_1(y) + d_2(y) - b^Ty, 
\end{equation}
where
\begin{equation}\label{eq:dxdy}
d_i(y) := \min_{x_{i}\in X_{i}}\left\{\phi_i(x_{i}) + y^TA_{i}x_{i} \right\}, ~i=1,2. 
\end{equation}
We denote by $x_{i}^{*}(y)$ a solution of the minimization problem in \eqref{eq:dxdy} ($i=1,2$) and $x^{*}(y) := (x^{*}_{1}(y), x^{*}_{2}(y))$.
Since $\phi_i$ is continuous and $X_{i}$ is closed and bounded, this problem has a solution.
Note that if $x_{i}^{*}(y)$ is not uniques for a given $y$ then $d_i$ is not differentiable at the point $y$ ($i=1,2$).
Consequently, $d$ is not differentiable at $y$. The representation \eqref{eq:dual_func2}-\eqref{eq:dxdy} is called a \textit{dual decomposition} of the dual
function $d$.

\subsection{Smoothing the dual function via prox-functions} 
By assumption that $X_{i}$ is bounded, instead of considering the nonsmooth function $d$, we smooth the dual function $d$ by means of prox-functions.
A function $p_i$ is called a proximity function (prox-function) of a given nonempty, closed and bounded convex set $X_i\subset\mathbb{R}^{n_i}$ if $p_i$ is
continuous, strongly convex with convexity parameter $\sigma_i > 0$ and $X_i \subseteq\mathrm{dom}(p_i)$.
 
Suppose that $p_i$ is a prox-function of $X_{i}$ and  $\sigma_i > 0$ is its convexity parameter ($i=1,2$).
Let us consider the following functions:
\begin{align}
&d_i(y; \beta_1) := \min_{x_{i}\in X_{i}}\left\{\phi_i(x_{i}) + y^TA_{i}x_{i} + \beta_1p_i(x_{i})\right\}, ~i=1,2, \label{eq:d_lbdxdy}\\
&d(y; \beta_1) := d_1(y; \beta_1) + d_2(y; \beta_1) - b^Ty. \label{eq:d_lbd} 
\end{align}
Here, $\beta_1>0$ is a given parameter called smoothness parameter.
We denote by $x_i^{*}(y;\beta_1)$  the solution of \eqref{eq:d_lbdxdy}, i.e.:
\begin{eqnarray}\label{eq:Txy}
&x^{*}_i(y; \beta_1) := \textrm{arg}\!\!\!\displaystyle\min_{x_{i}\in X_{i}}\left\{\phi_i(x_{i}) + y^TA_{i}x_{i} + \beta_1p_i(x_{i})\right\}, ~i=1, 2.
\end{eqnarray}
Note that it is possible to use different parameters $\beta_1^{i}$ for \eqref{eq:d_lbdxdy} ($i=1,2$).

Let $x_{i}^c$ be the prox-center  of $X_{i}$ which is defined as:
\begin{equation}\label{eq:x_0y_0}
x_{i}^c = \textrm{arg}\!\!\min_{x_{i}\in X_{i}}p_i(x_{i}), ~i=1, 2.
\end{equation}
Without loss of generality, we can assume that $p_i(x_{i}^c) = 0$.  Since $X_{i}$ is bounded, the quantity
\begin{equation}\label{eq:DxDy}
D_i := \max_{x_{i}\in X_{i}}p_i(x_{i})
\end{equation}
is well-defined and $0\leq D_i < +\infty$ for $i=1,2$.
The following lemma shows the main properties of $d(\cdot; \beta_1)$, whose proof can be found, e.g., in \cite{Necoara2008,Nesterov2005}.

\begin{lemma}\label{le:smoothing_estimate}
For any $\beta_1 > 0$, the function $d_i(\cdot; \beta_1)$ defined by \eqref{eq:d_lbdxdy} is well-defined and continuously differentiable on $\mathbb{R}^m$.
Moreover, this function is concave and its gradient w.r.t $y$ is given as: 
\begin{equation}\label{eq:grad_d_xdy}
\nabla d_i(y; \beta_1) = A_{i}x_{i}^{*}(y; \beta_1), ~i=1, 2, 
\end{equation}
which is Lipschitz continuous with a Lipschitz constant $L_i^{d}(\beta_1) = \frac{\norm{A_{i}}^2}{\beta_1\sigma_i}$ ($i=1,2$).
The following estimates hold:
\begin{equation}\label{eq:dxdy_estimate}
d_i(y; \beta_1) \geq d_i(y) \geq d_i(y; \beta_1) - \beta_1D_i, ~i=1,2. 
\end{equation}
Consequently, the function $d(\cdot;\beta_1)$ defined by \eqref{eq:d_lbd} is concave and differentiable and its gradient is given by
$\nabla{d}(y;\beta_1) :=
Ax^{*}(y;\beta_1) - b$ which is Lipschitz continuous with a Lipschitz constant $L^d(\beta_1) := \frac{1}{\beta_1}\sum_{i=1}^2\frac{\norm{A_i}^2}{\sigma_i}$.
Moreover, it holds that:
\begin{equation}\label{eq:d_estimate} 
d(y; \beta_1) \geq d(y) \geq d(y; \beta_1) - \beta_1(D_1 + D_2).
\end{equation}
\end{lemma}
The inequalities \eqref{eq:d_estimate} show that $d(\cdot; \beta_1)$ is an approximation of $d$. Moreover, $d(\cdot; \beta_1)$
converges to $d$ as $\beta_1$ tends to zero.

\begin{remark}\label{re:bound_XY}
Even without the assumption that $X$ is bounded, if the solution set $X^{*}$ of \eqref{eq:distNLP} is bounded then, in principle, we can bound the feasible set $X$ by a large compact set which contains all the sampling points generated by the
algorithms (see Section \ref{sec:decom_primal_dual_acceleration} below). 
However, in the following algorithms we do not use $D_i$, $i=1,2$ (defined by \eqref{eq:DxDy}) in any computational step. They only appear in the theoretical
complexity estimates. 
\end{remark}

Next, for a given $\beta_2 > 0$, we define a mapping $\psi(\cdot; \beta_2)$ from $X$ to $\mathbb{R}$ by:
\begin{equation}\label{eq:psi}
\psi(x; \beta_2) := \max_{y \in \mathbb{R}^m}\left\{ (Ax - b)^Ty - \frac{\beta_2}{2}\norm{y}^2 \right\}.
\end{equation}
This function can be considered as an approximate version of $\psi(x) := \displaystyle\max_{y\in\mathbb{R}^m}\left\{(Ax-b)^Ty \right\}$ using the prox-function $p(y) := \frac{1}{2}\norm{y}^2$.
It is easy to show that the unique solution of the maximization problem in \eqref{eq:psi} is given explicitly as $y^{*}(x;\beta_2) = \frac{1}{\beta_2}(Ax - b)$
and $\psi(x;\beta_2) = \frac{1}{2\beta_2}\norm{Ax-b}^2$.
Therefore,
$\psi(\cdot; \beta_2)$ is well-defined and differentiable on $X$.
Let
\begin{equation}\label{eq:f_beta2}
f(x; \beta_2) := \phi(x) + \psi(x; \beta_2) = \phi(x) + \frac{1}{2\beta_2}\norm{Ax-b}^2.
\end{equation}
The  next lemma summarizes the properties of $\psi(\cdot; \beta_2)$. 

\begin{lemma}\label{le:Lipschitz_diff_psi}
For any $\beta_2 > 0$, the function $\psi(\cdot;\beta_2)$ defined by \eqref{eq:psi} is continuously differentiable on $X$ and its gradient is given by:
\begin{equation}\label{eq:d_psi}
\nabla\psi(x; \beta_2) = (\nabla_{x_{1}}\psi(x; \beta_2), \nabla_{x_{2}}\psi(x;\beta_2)) = (A_{1}^Ty^{*}(x; \beta_2),~ A_{2}^Ty^{*}(x; \beta_2)),
\end{equation}
which is Lipschitz continuous with a Lipschitz constant $L^{\psi}(\beta_2) := \frac{1}{\beta_2}(\norm{A_{1}}^2 + \norm{A_{2}}^2)$.
Moreover, the following estimate holds for all $x, \hat{x} \in X$:
\begin{eqnarray}\label{eq:estimate_psi}
\psi(x; \beta_2) &&\leq \psi(\hat{x}; \beta_2) + \nabla_1\psi(\hat{x}; \beta_2)^T(x_{1}-\hat{x}_{1}) + \nabla_2\psi(\hat{x}; \beta_2)^T(x_{2}-\hat{x}_{2}) \nonumber\\
[-1.5ex]\\[-1.5ex]
&& + \frac{L_1^{\psi}(\beta_2)}{2}\norm{x_{1} \!-\! \hat{x}_{1}}^2 \!+\! \frac{L_2^{\psi}(\beta_2)}{2}\norm{x_{2} \!-\! \hat{x}_{2}}^2, \nonumber
\end{eqnarray}
where $L_1^{\psi}(\beta_2) := \frac{2}{\beta_2}\norm{A_{1}}^2$ and $L_2^{\psi}(\beta_2) := \frac{2}{\beta_2}\norm{A_{2}}^2$.
\end{lemma}

\begin{proof}
Since $\psi(x; \beta_2) = \frac{1}{2\beta_2}\norm{A_{1}x_{1} + A_{2}x_{2} - b}^2$ by the definition \eqref{eq:psi} and $y^{*}(x; \beta_2) = \frac{1}{\beta_2}(A_{1}x_{1} + A_{2}x_{2} -
b)$, it is easy to compute directly $\nabla\psi(\cdot; \beta_2)$. Moreover, we have:
\begin{eqnarray}\label{eq:convexity_psi}
\psi(x; \beta_2) \!-\! \psi(\hat{x}; \beta_2) \!-\! \nabla\psi(\hat{x}; \beta_2)^T(x \!-\! \hat{x})  \!\! && = \frac{1}{2\beta_2}\norm{A_{1}(x_{1} -\hat{x}_{1}) + A_{2}(x_{2}
-\hat{x}_{2})}^2 \nonumber\\
[-1.5ex]\\[-1.5ex]
&&\leq \frac{1}{\beta_2}\norm{A_{1}}^2\norm{x_{1} - \hat{x}_{1}}^2 + \frac{1}{\beta_2}\norm{A_{2}}^2\norm{x_{2} - \hat{x}_{2}}^2.\nonumber
\end{eqnarray}
This inequality is indeed \eqref{eq:estimate_psi}.
\Eproof
\end{proof}
From the definition of $f(\cdot; \beta_2)$, we obtain:
\begin{equation}\label{eq:smooth_f}
f(x;  \beta_2) - \frac{1}{2\beta_2}\norm{Ax - b}^2 = \phi(x) \leq f(x; \beta_2). 
\end{equation}
Note that $f(\cdot; \beta_2)$ is an upper bound of $\phi(\cdot)$ instead of a lower bound as in \cite{Nesterov2005}.
Note that the Lipschitz constants in \eqref{eq:estimate_psi} are roughly estimated. These quantities can be quantified carefully
by taking into account the problem structure to trade-off the computational effort in each component subproblem.

\subsection{Excessive gap technique}
Since the primal-dual gap of the primal and dual problems \eqref{eq:distNLP}-\eqref{eq:dual_prob} is measured by $ g(x,y) := \phi(x) - d(y)$, if the gap $g$ is equal to zero
for some feasible point $x$ and $y$ then this point is an optimal solution of \eqref{eq:distNLP}-\eqref{eq:dual_prob}. 
In this section, we apply to the Lagrangian dual decomposition framework a technique called \textit{excessive gap} proposed by Nesterov in \cite{Nesterov2005}.

Let us consider $\hat{d}(y; \beta_1) := d(y; \beta_1) - \beta_1(D_1 + D_2)$.
It follows from \eqref{eq:d_estimate} and \eqref{eq:smooth_f} that $\hat{d}(\cdot; \beta_1)$ is an underestimate of $d(\cdot)$, while  $f(\cdot; \beta_2)$ is an overestimate of $\phi(\cdot)$.
Therefore, $0 \leq g(x,y) = \phi(x) - d(y) \leq f(x; \beta_2) - d(y; \beta_1) + \beta_1(D_1 + D_2)$.
Let us recall the following excessive gap condition introduced in \cite{Nesterov2005}.
\begin{definition}\label{de:excessive_gap}
We say that a point $(\bar{x}, \bar{y})\in X\times\mathbb{R}^m$ satisfies the \textit{excessive gap} condition  with respect to two smoothness parameters
$\beta_1 > 0$ and $\beta_2 > 0$ if:
\begin{equation}\label{eq:excessive_gap}
f(\bar{x}; \beta_2) \leq d(\bar{y}; \beta_1),
\end{equation}
where $f(\cdot;\beta_2)$ and $d(\cdot;\beta_1)$ are defined by \eqref{eq:smooth_f} and \eqref{eq:d_lbd}, respectively.
\end{definition}
The following lemma provides an upper bound estimate for the duality gap and the feasibility gap of problem \eqref{eq:distNLP}.

\begin{lemma}\label{le:excessive_gap}
Suppose that $(\bar{x}, \bar{y}) \in X\times\mathbb{R}^m$ satisfies the excessive gap condition \eqref{eq:excessive_gap}. Then for any $y^{*}\in
Y^{*}$, we have:
\begin{eqnarray}
&-\norm{y^{*}}\norm{A\bar{x} \!-\! b} &\!\leq\! \phi(\bar{x}) \!-\! d(\bar{y}) \!\leq\! \beta_1(D_1 \!+\! D_2) \!-\! \frac{1}{2\beta_2}\norm{A\bar{x} \!-\! b}^2
\!\leq\! \beta_1(D_1 \!+\! D_2)\label{eq:dual_gap},\\
[-1.5ex]
\textrm{and} ~~&& \nonumber\\
[-1.5ex]
&\norm{A\bar{x} - b} &\leq \beta_2\left[\norm{y^{*}} + \sqrt{\norm{y^{*}}^2+\frac{2\beta_1}{\beta_2}(D_1 + D_2)}\right]. \label{eq:feasibility}
\end{eqnarray}
\end{lemma}

\begin{proof}
Suppose that $\bar{x}$ and $\bar{y}$ satisfy condition \eqref{eq:excessive_gap}.
For a given $y^{*}\in Y^{*}$, one has:
\begin{align*}
d(\bar{y})&\leq d(y^{*}) = \min_{x\in X}\left\{\phi(x) + (Ax-b)^Ty^{*}\right\} \leq \phi(\bar{x}) + (A\bar{x} - b)^Ty^{*}\notag\\
&\leq \phi(\bar{x}) + \norm{A\bar{x} - b}\norm{y^{*}}, 
\end{align*}
which implies the first inequality of \eqref{eq:dual_gap}.
By using Lemma \ref{le:smoothing_estimate} and \eqref{eq:f_beta2} we have:
\begin{align*}
\phi(\bar{x}) - d(\bar{y}) \overset{\tiny\eqref{eq:d_estimate}+\eqref{eq:smooth_f}}{\leq} f(\bar{x}; \beta_2) - d(\bar{y}; \beta_1) + \beta_1(D_1 + D_2) -
\frac{1}{2\beta_2}\norm{A\bar{x} - b}^2. 
\end{align*}
Now, by substituting the condition \eqref{eq:excessive_gap} into this inequality, we obtain the second inequality of \eqref{eq:dual_gap}.
Let $\eta := \norm{Ax - b}$. It follows from \eqref{eq:dual_gap} that $\eta^2 - 2\beta_2\norm{y^{*}}\eta - 2\beta_1\beta_2(D_1+D_2) \leq 0$.
The estimate \eqref{eq:feasibility} follows from this inequality after few simple calculations.
\Eproof
\end{proof}

\section{New decomposition algorithm}\label{sec:decom_primal_acceleration}
In this section, we derive an iterative decomposition algorithm for solving \eqref{eq:distNLP} based on the excessive gap technique. This method is called a \textit{decomposition algorithm with primal update}.
The aim is to generate a point $(\bar{x}, \bar{y})\in X\times\mathbb{R}^m$ at each iteration such that this point maintains the excessive gap condition \eqref{eq:excessive_gap} while the algorithm drives the parameters $\beta_1$ and $\beta_2$ to zero.

\subsection{Proximal mappings}
As assumed earlier, the function $\phi_i$ is convex but not necessarily differentiable. Therefore, we can not use the gradient information of these functions. 
We consider the following mappings ($i=1,2$):
\begin{align}
&P_i(\hat{x}; \beta_2) := \textrm{arg}\!\!\min_{x_{i}\in X_{i}}\left\{\phi_i(x_{i}) + y^{*}(\hat{x}; \beta_2)^TA_{i}(x_{i} - \hat{x}_{i}) +
\frac{L_i^{\psi}(\beta_2)}{2}\norm{x_{i} - \hat{x}_{i}}^2\right\}, \label{eq:prox_mapping_xy}
\end{align}
where $y^{*}(\hat{x}; \beta_2) := \frac{1}{\beta_2}(A\hat{x} - b)$.
Since $L_i^{\psi}(\beta_2)$ defined in Lemma \ref{le:Lipschitz_diff_psi} is positive, $P_i(\cdot; \beta_2)$ is well-defined. 
This mapping is called \textit{proximal operator} \cite{Chen1994}. Let $P(\cdot; \beta_2) = (P_1(\cdot; \beta_2), P_2(\cdot; \beta_2))$.

First, we state that the excessive gap condition \eqref{eq:excessive_gap} is well-defined by showing that there exists a point $(\bar{x},\bar{y})$ that satisfies \eqref{eq:excessive_gap}.
This point will be used as a starting point in Algorithm \ref{alg:A1} described below. 

\begin{lemma}\label{le:intial_point}
Suppose that $x^c = (x_{1}^c; x_{2}^c)$ is the prox-center of $X$. For a given $\beta_2 > 0$, let us define:
\begin{equation}\label{eq:initial_point}
\bar{y} := \frac{1}{\beta_2}(Ax^c - b) ~~\mathrm{and} ~~\bar{x} := P(x^c; \beta_2).
\end{equation}
If the parameter $\beta_1$ is chosen such that:
\begin{equation}\label{eq:initial_point_cond}
\beta_1\beta_2 \geq 2\max_{1\leq i\leq 2}\left\{\frac{\norm{A_{i}}^2}{\sigma_i}\right\}, 
\end{equation}
then $(\bar{x}$, $\bar{y})$ satisfies the excessive gap condition \eqref{eq:excessive_gap}.
\end{lemma}
The proof of Lemma \ref{le:intial_point} can be found in the appendix.

\subsection{Primal step}
Suppose that $(\bar{x}, \bar{y})\in X\times\mathbb{R}^m$ satisfies the excessive gap condition \eqref{eq:excessive_gap}. 
We generate a new point $(\bar{x}^{+}, \bar{y}^{+})\in X\times\mathbb{R}^m$ and by applying the following update scheme:
\begin{align}
&(\bar{x}^{+}, \bar{y}^{+}) := \mathcal{A}^p_m(\bar{x},\bar{y}; \beta_1,\beta_2^{+}, \tau) \Longleftrightarrow \begin{cases}
&\hat{x} := (1-\tau)\bar{x} + \tau x^{*}(\bar{y}; \beta_1),\\
&\bar{y}^{+} : = (1-\tau)\bar{y} + \tau y^{*}(\hat{x};\beta_2^{+}),\\
&\bar{x}^{+} := P(\hat{x}; \beta_2^{+}),
\end{cases}\label{eq:main_update_rule}\\
&\beta_1^{+} := (1-\tau)\beta_1 ~\textrm{and}~ \beta_2^{+} = (1-\tau)\beta_2, \label{eq:main_update_rule_beta}
\end{align}
where $P(\cdot; \beta_2^{+}) = (P_1(\cdot; \beta_2^{+}), P_2(\cdot; \beta_2^{+}))$ and $\tau \in (0,1)$ will be chosen appropriately. 

\begin{remark}\label{re:parallel}
In the scheme \eqref{eq:main_update_rule}, the points $x^{*}(\bar{y}; \beta_1) = (x^{*}_{1}(\bar{y}; \beta_1), x_{2}^{*}(\bar{y}; \beta_1))$, $\hat{x} =
(\hat{x}_{1},\hat{x}_{2})$ and $\bar{x}^{+} = (\bar{x}^{+}_{1}, \bar{x}^{+}_{2})$ can be computed \textit{in parallel}. To compute $x^{*}(\bar{y};\beta_1)$ and
$\bar{x}^{+}$ we need to solve the corresponding convex programs in $\mathbb{R}^{n_1}$ and $\mathbb{R}^{n_2}$, respectively.
\end{remark}
 
The following theorem shows that the update rule \eqref{eq:main_update_rule} maintains the excessive gap condition \eqref{eq:excessive_gap}.

\begin{theorem}\label{th:main_rule}
Suppose that $(\bar{x},\bar{y}) \in X\times\mathbb{R}^m$ satisfies \eqref{eq:excessive_gap} with respect to two values $\beta_1 > 0$ and $\beta_2 > 0$. Then
$(\bar{x}^{+}, \bar{y}^{+})$ generated by scheme \eqref{eq:main_update_rule}-\eqref{eq:main_update_rule_beta} is in $X\times\mathbb{R}^m$ and maintains the
excessive gap condition \eqref{eq:excessive_gap} with respect to two smoothness parameter values  $\beta_1^{+}$ and $\beta_2^{+}$ provided that:
\begin{equation}\label{eq:main_condition}
\beta_1\beta_2 \geq \frac{2\tau^2}{(1-\tau)^2}\max_{1\leq i\leq 2}\left\{\frac{\norm{A_{i}}^2}{\sigma_i}\right\}. 
\end{equation}
\end{theorem}
 
\begin{proof}
The last line of \eqref{eq:main_update_rule} shows that $\bar{x}^{+}\in X$.
Let us denote by $\hat{y} = y^{*}(\hat{x}; \beta_2^{+})$. 
Then, by using the definition of $d(\cdot; \beta_1)$, the second line of \eqref{eq:main_update_rule} and $\beta_1^{+} = (1-\tau)\beta_1$, we have:
\begin{eqnarray}\label{eq:th31_est1}
d(\bar{y}^{+}; \beta_1^{+}) \!\! &&= \min_{x\in X}\left\{\phi(x) + (Ax-b)^T\bar{y}^{+} + \beta_1^{+}[p_1(x_{1})+p_2(x_{2})]\right\} \nonumber\\
&&\overset{\scriptsize\textrm{line $2$ \eqref{eq:main_update_rule}}}{=}\min_{x\in X}\left\{ \phi(x) + (1-\tau)(Ax-b)^T\bar{y} + \tau(Ax-b)^T\hat{y}\right. \nonumber\\ 
&& + \left. (1-\tau)\beta_1[p_1(x_{1}) + p_2(x_{2})]\right\} \\
&& = \min_{x\in X}\left\{ (1-\tau)\left[\phi(x) + (Ax-b)^T\bar{y} + \beta_1[p_1(x_{1})+p_2(x_{2})]\right]\right.\nonumber\\
&& + \left. \tau\left[\phi(x) + (Ax-b)^T\hat{y}\right]\right\}.\nonumber 
\end{eqnarray}
Now, we estimate the first term in the last line of \eqref{eq:th31_est1}.
Since $\beta_2^{+} = (1-\tau)\beta_2$, one has:
\begin{align}\label{eq:th31_est1b}
\psi(\bar{x}; \beta_2) &= \frac{1}{2\beta_2}\norm{A\bar{x} - b}^2 = (1-\tau)\frac{1}{2\beta^{+}_2}\norm{A\bar{x} - b}^2 = (1-\tau)\psi(\bar{x}; \beta_2^{+}). 
\end{align}
Moreover, if we denote by $x^1 = x^{*}(\bar{y}; \beta_1)$ then, by the strong convexity of $p_1$ and $p_2$, \eqref{eq:th31_est1b} and $f(\bar{x}; \beta_2) \leq
d(\bar{y}; \beta_1)$, we have:
\begin{eqnarray}\label{eq:th31_est2}
T_1 \!\! && := \phi(x) + (Ax-b)^T\bar{y} + \beta_1[p_1(x_{1}) + p_2(x_{2})] \nonumber\\
&& \geq \min_{x\in X}\left\{\phi(x) + (Ax \!-\! b)^T\bar{y} \!+\! \beta_1[p_1(x_{1}) \!+\! p_2(x_{2})]\right\}  \!+\! \frac{1}{2}\beta_1\!\!\left[\sigma_1\norm{x_{1} \!-\! x_{1}^1}^2 \!+\! \sigma_2\norm{x_{2} \!-\! x_{2}^1}^2\!\right]\nonumber\\
&& = d(\bar{y}; \beta_1) + \frac{1}{2}\beta_1\left[\sigma_1\norm{x_{1} - x_{1}^1}^2 + \sigma_2\norm{x_{2} - x_{2}^1}^2\right]\nonumber\\
[-1.5ex]\\[-1.5ex]
&& \overset{\scriptsize\eqref{eq:excessive_gap}}{\geq} f(\bar{x}; \beta_2) +  \frac{1}{2}\beta_1\left[\sigma_1\norm{x_{1} - x_{1}^1}^2 + \sigma_2\norm{x_{2} - x_{2}^1}^2\right]\nonumber\\
&& \overset{\scriptsize\textrm{def.~}f(\cdot; \beta_2)}{=} \phi(\bar{x}) + \psi(\bar{x}; \beta_2) +  \frac{1}{2}\beta_1\left[\sigma_1\norm{x_{1} - x_{1}^1}^2 + \sigma_2\norm{x_{2} - x_{2}^1}^2\right]\nonumber\\
&& \overset{\scriptsize\eqref{eq:th31_est1b}}{=} \phi(\bar{x}) + \psi(\bar{x}; \beta_2^{+}) +  \frac{1}{2}\beta_1\left[\sigma_1\norm{x_{1} - x_{1}^1}^2 + \sigma_2\norm{x_{2} - x_{2}^1}^2\right] -
\tau\psi(\bar{z};\beta_2^{+})\nonumber\\
&& \overset{\scriptsize\eqref{eq:convexity_psi}}{=} \phi(\bar{x}) + \psi(\hat{x}; \beta_2^{+}) + \nabla\psi(\hat{x}; \beta_2^{+})^T(\bar{x} -\hat{x}) + \frac{1}{2}\beta_1\left[\sigma_1\norm{x_{1} - x_{1}^1}^2 + \sigma_2\norm{x_{2} - x_{2}^1}^2\right]\nonumber\\
&& + \frac{1}{2\beta_2^{+}}\norm{A(\bar{x}-\hat{x})}^2 - \tau\psi(\bar{x}; \beta_2^{+}).\nonumber  
\end{eqnarray}
For the second term in the last line of \eqref{eq:th31_est1}, we use the fact that $\hat{y} = \frac{1}{\beta^{+}_2}(A\hat{x} -
b)$ and $\nabla_y\psi(\hat{x};\beta_2) = A^T\hat{y}$ to obtain:
\begin{eqnarray}\label{eq:th31_est3}
T_2 &&:= \phi(x) + (Ax-b)^T\hat{y}\nonumber\\
&& = \phi(x) + \hat{y}^TA(x-\hat{x}) + (A\hat{x} - b)^T\hat{y} \nonumber\\
[-1.5ex]\\[-1.5ex]
&& \overset{\tiny\textrm{def.~}\hat{y}+\eqref{eq:d_psi}}{=} \phi(x) + \nabla\psi(\hat{x}; \beta_2^{+})^T(x - \hat{x}) +
\frac{1}{\beta_2^{+}}\norm{A\hat{x} - b}^2 \nonumber\\
&& \overset{\tiny\textrm{def.~}\hat{\psi}}{=} \phi(x) + \psi(\hat{x}; \beta_2^{+}) + \nabla\psi(\hat{x}; \beta_2^{+})^T(x - \hat{x}) +
\psi(\hat{x}; \beta_2^{+}). \nonumber
\end{eqnarray}
Substituting \eqref{eq:th31_est2} and \eqref{eq:th31_est3} into \eqref{eq:th31_est1} and noting that $(1-\tau)(\bar{x}-\hat{x}) + \tau(x - \hat{x}) = \tau(x-
x^1)$ due to the first line of \eqref{eq:main_update_rule}, we obtain:
\begin{eqnarray}\label{eq:th31_est4}
d(\bar{y}^{+}; \beta_1^{+}) && = \min_{x\in X}\left\{ (1-\tau)T_1 + \tau T_2 \right\} \nonumber\\
&& \overset{\tiny\eqref{eq:th31_est2}+\eqref{eq:th31_est3}}{\geq} \min_{x\in X}\Big\{ (1-\tau)\Big[\phi(\bar{x}) + \psi(\hat{x}; \beta_2^{+}) +
\nabla\psi(\hat{x}; \beta_2^{+})^T(\bar{x}-\hat{x}) \nonumber\\
&& + \frac{1}{2}\beta_1\left[\sigma_1\norm{x_{1}-x_{1}^1}^2 + \sigma_2\norm{x_{2} - x_{2}^1}^2 \right] \Big] \nonumber\\
&& + \tau\left[\phi(x) + \psi(\hat{x}; \beta_2^{+}) + \nabla\psi(\hat{x}; \beta_2^{+})^T(x-\hat{x})\right]\Big\} \nonumber\\
&& - \tau(1-\tau)\psi(\bar{x}; \beta_2^{+}) + \frac{(1-\tau)}{2\beta_2^{+}}\norm{A(\bar{x}-\hat{x})}^2 + \tau\psi(\hat{x}; \beta^{+}_2) \\
&& = \min_{x\in X}\Big\{ (1-\tau)\phi(\bar{x}) + \tau\phi(x) + \psi(\hat{x}; \beta_2^{+}) + \nabla\psi(\hat{x}; \beta_2^{+})^T\left[(1-\tau)(\bar{x}-\hat{x}) +
\tau(x-\hat{x})\right] \nonumber\\
&& + \frac{1}{2}(1-\tau)\beta_1\left[\sigma_1\norm{x_{1}-x_{1}^1}^2 + \sigma_2\norm{x_{2} - x_{2}^1}^2 \right] \Big\} + \mathrm{\textbf{T}}_3\nonumber\\
&& \overset{\tiny\phi-\mathrm{convex}}{\geq} \min_{x\in X}\Big\{ \phi((1-\tau)\bar{x} + \tau x) + \psi(\hat{x}; \beta_2^{+}) + \tau\nabla\psi(\hat{x};
\beta_2^{+})^T(x-x^1) \nonumber\\
&& + \frac{1}{2}(1-\tau)\beta_1\left[\sigma_1\norm{x_{1}-x_{1}^1}^2 + \sigma_2\norm{x_{2} - x_{2}^1}^2 \right] \Big\} + \mathrm{\textbf{T}}_3,\nonumber
\end{eqnarray}
where $\mathrm{\textbf{T}}_3 := \frac{(1-\tau)}{2\beta_2^{+}}\norm{A(\bar{x}-\hat{x})}^2 + \tau\psi(\hat{x}; \beta^{+}_2) - \tau(1-\tau)\psi(\bar{x};
\beta_2^{+})$.
Next, we note that the condition \eqref{eq:main_condition} is equivalent to:
\begin{equation}\label{eq:main_condition_tmp}
(1-\tau)\beta_1\sigma_i \geq \frac{2\tau^2}{(1-\tau)\beta_2}\norm{A_i}^2 \geq L_i^{\psi}(\beta_2^{+})\tau^2, ~ i = 1,2. 
\end{equation}
Moreover, if we denote by $u := \bar{x} + \tau(x - \bar{x})$ then:
\begin{equation}\label{eq:th31_est4c}
u - \hat{x} = \bar{x} + \tau(x - \bar{x}) - \hat{x}  = \bar{x} + \tau(x - \bar{x}) - (1-\tau)\bar{x} - \tau x^1 = \tau(x - x^1).
\end{equation}
Now, by using Lemma \ref{le:Lipschitz_diff_psi}, the condition \eqref{eq:main_condition_tmp} and \eqref{eq:th31_est4c}, the
estimation \eqref{eq:th31_est4}
becomes:
\begin{eqnarray}\label{eq:th31_est5}
d(\bar{y}^{+}; \beta_1^{+}) - \mathrm{\textbf{T}}_3 \!\! &&\overset{\tiny\eqref{eq:th31_est4c}}{\geq} \min_{u := \bar{x}+\tau(x-\bar{x})\in \bar{x} +
\tau(X-\bar{x})}\Big\{ \phi(u) + \psi(\hat{x}; \beta_2^{+}) + \nabla\psi(\hat{x}; \beta_2)^T(u-\hat{x}) \nonumber\\
&& + \frac{\beta_1(1-\tau)\sigma_1}{2\tau^2}\norm{u_{1} -\hat{x}_{1}}^2 + \frac{\beta_1(1-\tau)\sigma_2}{2\tau^2}\norm{u_{2} - \hat{x}_{2}}^2 \Big\} \nonumber\\
&& \overset{\tiny{\bar{x} + \tau(X-\bar{x}) \subseteq X}}{\geq} \min_{u\in X}\Big\{\psi(\hat{x}; \beta_2^{+}) + \phi(u) + \nabla\psi(\hat{x};
\beta_2^{+})^T(u-\hat{x}) \\
&&  + \frac{\beta_1(1-\tau)\sigma_1}{2\tau^2}\norm{u_{1} -\hat{x}_{1}}^2 + \frac{\beta_1(1-\tau)\sigma_2}{2\tau^2}\norm{u_{2} - \hat{x}_{2}}^2 \Big\}
\nonumber\\
&& \overset{\tiny\eqref{eq:main_condition_tmp}}{\geq} \min_{u\in X}\Big\{ \phi(u) + \psi(\hat{x};
\beta_2^{+}) + \nabla\psi(\hat{x}; \beta_2^{+})^T(u - \hat{x}) \nonumber\\
&& + \frac{L_1^{\psi}(\beta_2^{+})}{2}\norm{u_1 - \hat{x}_1}^2 + \frac{L_2^{\psi}(\beta_2^{+})}{2}\norm{u_2 - \hat{x}_2}^2\Big\} \nonumber\\
&& \overset{\tiny\textrm{line}~3~\eqref{eq:main_update_rule}}{=} \phi(\bar{x}^{+}) + \psi(\hat{x}; \beta_2^{+}) + \nabla\psi(\hat{x}; \beta_2^{+})^T(\bar{x}^{+}-\hat{x}) \nonumber\\
&& + \frac{L_1^{\psi}(\beta_2^{+})}{2}\norm{\bar{x}_{1}^{+} -\hat{x}_{1}}^2 + \frac{L_2^{\psi}(\beta_2^{+})}{2}\norm{\bar{x}_{2}^{+} - \hat{x}_{2}}^2\nonumber\\
&& \overset{\tiny\eqref{eq:estimate_psi}}{\geq} \phi(\bar{x}^{+}) + \psi(\bar{x}^{+}; \beta_2^{+}) =  f(\bar{x}^{+}; \beta_2^{+}). \nonumber 
\end{eqnarray}
To complete the proof, we show that $\mathrm{\textbf{T}}_3\geq 0$. Indeed, let us define $\hat{u} := A\hat{x} - b$ and $\bar{u} := A\bar{x} - b$, then $\hat{u}
- \bar{u} =
A(\hat{x}-\bar{x})$. We have:
\begin{eqnarray}\label{eq:th31_est6}
\mathrm{\textbf{T}}_3  && \overset{\tiny\textrm{def.}~\psi(\cdot; \beta_2)}{=} \frac{\tau}{2\beta_2^{+}}\norm{A\hat{x} - b}^2 -
\frac{\tau(1-\tau)}{2\beta_2^{+}}\norm{A\bar{x} - b}^2 + \frac{(1-\tau)}{2\beta_2^{+}}\norm{A(\hat{x}-\bar{x})}^2\nonumber\\
&& = \frac{1}{2\beta_2^{+}}\left[\tau\norm{\hat{u}}^2 - \tau(1-\tau)\norm{\bar{u}}^2 + (1-\tau)\norm{\hat{u}-\bar{u}}^2\right]  \nonumber\\
&& = \frac{1}{2\beta_2^{+}}\left[\tau\norm{\hat{u}}^2 - \tau(1-\tau)\norm{\bar{u}}^2 + (1-\tau)\norm{\hat{u}}^2 + (1-\tau)\norm{\bar{u}}^2 - 2(1-\tau)\hat{u}^T\bar{u}\right] \\
&& = \frac{1}{2\beta_2^{+}}\left[\norm{\hat{u}}^2 + (1-\tau)^2\norm{\bar{u}}^2 - 2(1-\tau)\hat{u}^T\bar{u}\right] \nonumber\\
&& = \frac{1}{2\beta_2^{+}}\norm{\hat{u} - (1-\tau)\bar{u}}^2 \geq 0. \nonumber
\end{eqnarray}
Substituting \eqref{eq:th31_est6} into \eqref{eq:th31_est5} we obtain the inequality $d(\bar{y}^{+}; \beta_1^{+}) \geq f(\bar{x}^{+}; \beta_2^{+})$.
\Eproof
\end{proof}

\begin{remark}\label{re:Lipschit_diff_phi}
If $\phi_i$ is convex and differentiable and its gradient is Lipschitz continuous with a Lipschitz constant $L^{\phi_i}_i \geq 0$ for some $i=1,2$,
then instead of using the proximal mapping $P_i(\cdot; \beta_2)$ in \eqref{eq:main_update_rule} we can use the gradient mapping which is defined as:
\begin{align}
&G_i(\hat{x}; \beta_2^{+}) \!:=\! \textrm{arg}\!\!\!\min_{x_{i}\in X_{i}}\!\!\Big\{ \!\nabla\phi_i(\hat{x}_{i})^T\!\!(x_{i} \!-\! \hat{x}_{i}) \!+\!
y^{*}(\hat{x};
\beta_2)^TA_{i}(x_{i} \!-\! \hat{x}_{i}) \!+\!
\frac{\hat{L}_i^{\psi}(\beta_2^{+})}{2}\norm{x_{i}-\hat{x}_{i}}^2\Big\}, \label{eq:G_XY} 
\end{align}
where $\hat{L}_i^{\psi}(\beta_2^{+}) := L_{\phi_i} + \frac{2\norm{A_i}^2}{\beta^{+}_2}$.
Indeed, let us prove the condition $d(\bar{y}^{+};\beta_1^{+}) \geq f(\hat{\bar{x}}^{+};\beta_2^{+})$, where $G(x; \beta_2) := (G_1(x_1;\beta_2),
G_2(x_2;\beta_2))$ and $\hat{\bar{x}}^{+} := G(\hat{x};\beta_2^{+})$.
First, by using the convexity of $\phi_i$ and the Lipschitz continuity of its gradient, we have:
\begin{equation}\label{eq:rm3_est1}
\phi_i(\hat{x}_i) + \nabla\phi_i(\hat{x}_i)^T(u_i - \hat{x}_i) \leq \phi_i(u_i) \leq \phi_i(\hat{x}_i) + \nabla\phi_i(\hat{x}_i)^T(u_i - \hat{x}_i) +
\frac{L_{\phi_i}}{2}\norm{u_i-\hat{x}_i}^2.
\end{equation}
Next, by summing up the second inequality from $i=1$ to $2$ and adding to \eqref{eq:estimate_psi} we have:
\begin{eqnarray}\label{eq:rm3_est2}
\phi(u) + \psi(u; \beta_2^{+}) &&\leq \phi(\hat{x}) + \psi(\hat{x};\beta_2^{+}) + \left[\nabla\phi(\hat{x}) + \nabla\psi(\hat{x};\beta_2^{+})\right]^T(u -
\hat{x}) \nonumber\\
[-1.5ex]\\[-1.5ex]
&& +  \frac{\hat{L}_1^{\psi}(\beta_2^{+})}{2}\norm{u_1 -\hat{x}_1}^2 +  \frac{\hat{L}_2^{\psi}(\beta_2^{+})}{2}\norm{u_2 - \hat{x}_2}^2. \nonumber
\end{eqnarray}
Finally, from the second inequality of \eqref{eq:th31_est5} we have:
\begin{eqnarray}\label{eq:th31_est5_extra}
d(\bar{y}^{+}; \beta_1^{+}) - \mathrm{\textbf{T}}_3  && \overset{\tiny\eqref{eq:main_condition_tmp}}{\geq} \min_{u\in X}\Big\{ \phi(u) + \psi(\hat{x};
\beta_2^{+}) + \nabla\psi(\hat{x}; \beta_2^{+})^T(u - \hat{x}) \nonumber\\
&& + \frac{(1-\tau)\beta_1\sigma_1}{2\tau^2}\norm{u_1 - \hat{x}_1}^2 + \frac{(1-\tau)\beta_1\sigma_2}{2\tau^2}\norm{u_2 - \hat{x}_2}^2\Big\} \nonumber\\
&& \overset{\tiny\phi-\mathrm{convex}+\eqref{eq:rm3_est2}}{\geq} \min_{u\in X}\Big\{ \phi(\hat{x}) + \nabla\phi(\hat{x})^T(u-\hat{x}) + \psi(\hat{x};
\beta_2^{+}) + \nabla\psi(\hat{x}; \beta_2^{+})^T(u - \hat{x}) \nonumber\\
&& + \frac{\hat{L}_1^{\psi}(\beta_2^{+})}{2}\norm{u_1 - \hat{x}_1}^2 + \frac{\hat{L}_2^{\psi}(\beta_2^{+})}{2}\norm{u_2 - \hat{x}_2}^2\Big\} \nonumber\\
&& \overset{\tiny\eqref{eq:G_XY}}{=} \phi(\hat{x}) + \psi(\hat{x}; \beta_2^{+}) + \left[\nabla{\phi}(\hat{x}) + \nabla\psi(\hat{x};
\beta_2^{+})\right]^T(\hat{\bar{x}}^{+}-\hat{x}) \nonumber\\
&& + \frac{\hat{L}_1^{\psi}(\beta_2^{+})}{2}\norm{\hat{\bar{x}}_{1}^{+} -\hat{x}_{1}}^2 + \frac{\hat{L}_2^{\psi}(\beta_2^{+})}{2}\norm{\hat{\bar{x}}_{2}^{+} -
\hat{x}_{2}}^2\nonumber\\
&& \overset{\tiny\eqref{eq:rm3_est2}}{\geq} \phi(\hat{\bar{x}}^{+}) + \psi(\hat{\bar{x}}^{+}; \beta_2^{+}) =  f(\hat{\bar{x}}^{+}; \beta_2^{+}).
\nonumber 
\end{eqnarray}
In this case, the conclusion of Theorem \ref{th:main_rule} is still valid for the substitution $\hat{\bar{x}}^{+} := G(\hat{x}; \beta_2^{+})$ provided that:
\begin{equation}\label{eq:main_condition_diff}
\frac{(1-\tau)}{\tau^2}\beta_1\sigma_i \geq L_{\phi_i} + \frac{2\norm{A_i}^2}{(1-\tau)\beta_2}, ~i=1,2.
\end{equation}
If $X_i$ is polytopic then problem \eqref{eq:G_XY} becomes a convex quadratic programming problem.
\end{remark}

Now, let us show how to update the parameter $\tau$ such that the condition \eqref{eq:main_condition} holds for $\beta_1^{+}$ and $\beta_2^{+}$.
From the update rule \eqref{eq:main_update_rule_beta} we have $\beta_1^{+}\beta_2^{+} = (1-\tau)^2\beta_1\beta_2$.
Suppose that $\beta_1$ and $\beta_2$ satisfy the condition \eqref{eq:main_condition}, i.e.:
\begin{equation*}
\beta_1\beta_2 \geq \frac{\tau^2}{(1-\tau)^2}\bar{L}, ~\textrm{where}~ \bar{L} : =
2\max_{1\leq i\leq 2}\left\{\frac{\norm{A_{i}}^2}{\sigma_i}\right\}.  
\end{equation*}
If we substitute $\beta_1$ and $\beta_2$ by $\beta_1^{+}$ and $\beta_2^{+}$, respectively, in this inequality then we have $\beta_1^{+}\beta_2^{+} \geq
\frac{\tau_{+}^2}{(1-\tau_{+})^2}\bar{L}$. However, since $\beta_1^{+}\beta_2^{+} = (1-\tau)^2\beta_1\beta_2$, it implies $\beta_1\beta_2 \geq
\frac{\tau_{+}^2}{(1-\tau)^2(1-\tau_{+})^2}\bar{L}$.
Therefore, if $\frac{\tau^2}{(1-\tau)^2} \geq \frac{\tau_{+}^2}{(1-\tau)^2(1-\tau_{+})^2}$ then $\beta_1^{+}$ and $\beta_2^{+}$ satisfy
\eqref{eq:main_condition}. This condition leads to $\tau \geq \frac{\tau_{+}}{1-\tau_{+}}$. Since $\tau, \tau_{+}\in (0, 1)$, the last inequality implies $0 <
\tau_{+} < \frac{1}{2}$ and
\begin{equation}\label{eq:tau_condition}
0 < \tau_{+} \leq \frac{\tau}{\tau + 1} < 1.
\end{equation}
Hence, \eqref{eq:main_update_rule}-\eqref{eq:main_update_rule_beta} are well-defined.

Now, we define a rule to update the step size parameter $\tau$.
\begin{lemma}\label{le:choice_of_tau}
Suppose that $\tau_0$ is arbitrarily chosen in $(0, \frac{1}{2})$. Then the sequence $\{\tau_k\}_{k\geq 0}$ generated by:
\begin{equation}\label{eq:tau_sequence}
\tau_{k+1} := \frac{\tau_k}{\tau_k + 1} 
\end{equation}
satisfies the following equality: 
\begin{equation}\label{eq:tau_estimation}
\tau_k = \frac{\tau_0}{1 + \tau_0k},~~\forall k\geq 0.
\end{equation}
Moreover, the sequence $\{\beta_k\}_{k\geq 0}$ generated by $\beta_{k+1} = (1-\tau_k)\beta_k$ for fixed $\beta_0 > 0$ satisfies:
\begin{equation}\label{eq:beta_estimation}
\beta_k = \frac{\beta_0}{\tau_0k + 1}, ~~\forall k\geq 0.
\end{equation} 
\end{lemma}

\begin{proof}
If we denote by $t := \frac{1}{\tau}$ and consider the function $\xi(t) := t + 1$ then the sequence $\{t_k\}_{k\geq 0}$ generated by the rule $t_{k+1} :=
\xi(t_k) = t_k + 1$ satisfies $t_k = t_0 + k$ for all $k\geq 0$. Hence $\tau_k = \frac{1}{t_k} = \frac{1}{t_0 + k} = \frac{\tau_0}{\tau_0k + 1}$ for $k\geq 0$.
To prove \eqref{eq:beta_estimation}, we observe that $\beta_{k+1} = \beta_0\prod_{i=0}^k(1-\tau_i)$.
Hence, by substituting \eqref{eq:tau_estimation} into the last equality and carrying out a simple calculations, we get \eqref{eq:beta_estimation}.
\Eproof
\end{proof}

\begin{remark}\label{re:choice_of_tau}
Since $\tau_0\in (0, 0.5)$, from Lemma \ref{le:choice_of_tau} we see that with $\tau_0 \to 0.5^{-}$ (e.g., $\tau_0 = 0.499$) the right-hand side estimate of
\eqref{eq:beta_estimation} is minimized.
\end{remark}

\subsection{The algorithm and its worst case complexity}
Before presenting the algorithm, we assume that the prox-center $x^c_{i}$ of $X_{i}$ is given a priori for ($i=1,2$).
Moreover, the parameter sequence $\{\tau_k\}$ is updated by \eqref{eq:tau_sequence}.
The algorithm is presented in detail as follows:

\noindent\rule[1pt]{\textwidth}{1.0pt}{~~}
\begin{algorithm}\vskip -0.2cm\label{alg:A1}{~}(\textit{Decomposition Algorithm with Primal Update})
\end{algorithm}
\vskip -0.3cm
\noindent\rule[1pt]{\textwidth}{0.5pt}
\noindent\textbf{Initialization:}
\begin{enumerate}
\item Set $\tau_0 := 0.499$. Choose $\beta_1^0 > 0$ and $\beta_2^0 > 0$ as follows:
\begin{equation*}
\beta_1^0 = \beta_2^0 := \sqrt{2\max_{1\leq i\leq 2}\left\{ \frac{\norm{A_{i}}^2}{\sigma_i}\right\}}.
\end{equation*}
\item Compute $\bar{x}^0$ and $\bar{y}^0$ from \eqref{eq:initial_point} as:
\begin{equation*}
\bar{y}^0 := \frac{1}{\beta_2^0}(Ax^c - b) ~\mathrm{and}~ \bar{x}^0 := P(x^c; \beta_2^0),
\end{equation*}
\end{enumerate} 
\noindent\textbf{Iteration: }\texttt{For} $k=0,1,\cdots$ \texttt{do}
\begin{enumerate}
\item If a given stopping criterion is satisfied then terminate.
\item Update the smoothness parameter $\beta_2^{k+1} := (1-\tau_k)\beta_2^k$. 
\item Compute $\bar{x}^{k+1}_{i}$ \textit{in parallel} for $i=1,2$ and $\bar{y}^{k+1}$ by the scheme \eqref{eq:main_update_rule}:
\begin{equation*}\label{eq:step_2}
(\bar{x}^{k+1}, \bar{y}^{k+1}) := \mathcal{A}^p_m(\bar{x}^k, \bar{y}^k; \beta_1^k, \beta_2^{k+1}, \tau_k). 
\end{equation*}
\item Update the smoothness parameter: $\beta_1^{k+1} := (1-\tau_k)\beta_1^k$. 
\item Update the step size parameter $\tau_k$ by: $\tau_{k+1} := \frac{\tau_k}{\tau_k + 1}$.
\end{enumerate}
\texttt{End of For}.
\vskip-0.2cm
\noindent\rule[1pt]{\textwidth}{1.0pt}
As mentioned in Remark \ref{re:parallel}, there are two steps of the scheme $\mathcal{A}^p_m$ at Step 3 of Algorithm \ref{alg:A1} that can be parallelized. The
first step is finding $x^{*}(\bar{y}^k; \beta_1)$ and the second one is computing $\bar{x}^{k+1}$. 
In general, both steps require solving two convex programming problems in parallel. 
The stopping criterion of Algorithm \ref{alg:A1} at Step 1 will be discussed in Section \ref{sec:discussion}. 

The following theorem provides the worst-case complexity estimate for Algorithm \ref{alg:A1}.
\begin{theorem}\label{th:convergence}
Let $\{ (\bar{x}^k,\bar{y}^k) \}$ be a sequence generated by Algorithm \ref{alg:A1}. Then the following duality gap and feasibility gap hold:
\begin{eqnarray}
&\phi(\bar{x}^k) - d(\bar{y}^k) &\leq \frac{\sqrt{\bar{L}}(D_1 + D_2)}{0.499k+1}, \label{eq:duality_gap}\\
[-1.5ex]\textit{and} ~~~~~~~~&\nonumber\\[-1.5ex]
&\norm{A\bar{x}^k - b} &\leq \frac{\sqrt{\bar{L}}}{0.499k + 1}\left[\norm{y^{*}} + \sqrt{\norm{y^{*}}^2 + 2(D_1 + D_2)}\right],\label{eq:feasible_gap}
\end{eqnarray}
where $\bar{L} := 2\displaystyle\max_{1\leq i\leq 2}\left\{ \frac{\norm{A_{i}}^2}{\sigma_i}\right\}$ and $y^{*}\in Y^{*}$.
\end{theorem}

\begin{proof}
By the choice of $\beta_1^0 = \beta_2^0 = \sqrt{\bar{L}}$ and Steps 1 in the initialization phase of Algorithm \ref{alg:A1} we see that $\beta_1^k = \beta_2^k$
for all $k \geq 0$.
Moreover, since $\tau_0 = 0.499$, by Lemma \ref{le:choice_of_tau}, we have $\beta_1^k = \beta_2^k = \frac{\beta_0}{\tau_0k + 1}  =
\frac{\sqrt{\bar{L}}}{0.499k + 1}$.
Now, by applying Lemma \ref{le:excessive_gap} with $\beta_1$ and $\beta_2$ equal to $\beta_1^k$ and $\beta_2^k$ respectively, we obtain the estimates
\eqref{eq:duality_gap} and \eqref{eq:feasible_gap}.
\Eproof
\end{proof}

\begin{remark}\label{re:alg1}
The worst case complexity of Algorithm \ref{alg:A1} is $O(\frac{1}{\varepsilon})$. However, the constants in the estimations \eqref{eq:duality_gap} and \eqref{eq:feasible_gap} also depend on the choices of
$\beta_1^0$ and $\beta_2^0$, which satisfy the condition \eqref{eq:initial_point_cond}.
The values of $\beta_1^0$ and $\beta_2^0$ will affect the accuracy of the duality and feasibility gaps. 
\end{remark}
In Algorithm \ref{alg:A1} we can use a simple update rule $\tau_k = \frac{a}{k+1}$, where $a > 0$ is arbitrarily chosen such that the condition $\tau_{k+1} \leq
\frac{\tau_k}{\tau_k + 1}$ holds. However, the rule \eqref{eq:tau_sequence} is the tightest one.

\section{Switching decomposition algorithm}\label{sec:decom_primal_dual_acceleration}
In this section, we apply the switching strategy to obtain a new variant of the first algorithm proposed in
\cite[Algorithm 1]{Nesterov2005} for solving problem \eqref{eq:distNLP}.
This scheme alternately switches between the primal and dual step depending on the iteration counter $k$ being even or odd.
Apart from its application to Lagrangian dual decomposition, this variant is still different from the one in \cite{Nesterov2005} at two points. First, since we
assume that the objective function is not necessarily smooth, instead of using the gradient mapping in the primal scheme, we use the proximal mapping defined
by \eqref{eq:prox_mapping_xy} to construct the primal step.
In contrast, since the objective function in the dual scheme is Lipschitz continuously differentiable, we can directly use the gradient mapping to compute
$\bar{y}^{+}$ (see \eqref{eq:dual_step}).  
Second, we use the exact update rule for $\tau$ instead of the simplified one as in \cite{Nesterov2005}.

\subsection{The gradient mapping of the smoothed dual function}
Since the smoothed dual function $d(\cdot; \beta_1)$ is Lipschitz continuously differentiable on $\mathbb{R}^m$ (see Lemma \ref{le:smoothing_estimate}). We
define the following mapping:
\begin{equation}\label{eq:dual_prox_mapping}
G(\hat{y}; \beta_1) := \textrm{arg}\!\max_{y\in\mathbb{R}^m}\left\{\nabla d(\hat{y}; \beta_1)^T(y -\hat{y}) -
\frac{L^{d}(\beta_1)}{2}\norm{y-\hat{y}}^2\right\},  
\end{equation}
where $L^{d}(\beta_1) := L_1^{d}(\beta_1) + L_2^{d}(\beta_1) = \frac{\norm{A_{1}}^2}{\beta_1\sigma_1} + \frac{\norm{A_{2}}^2}{\beta_1\sigma_2}$ and $\nabla
d(\hat{y}; \beta_1) = A_{1}x_{1}^{*}(\hat{y}; \beta_1) + A_{2}x_{2}^{*}(\hat{y}; \beta_1) - b$. 
This problem can explicitly be solved to get the unique solution:
\begin{equation}\label{eq:dual_prox_mapping_ex}
G(\hat{y}; \beta_1) =  \frac{1}{L^{d}(\beta_1)}[Ax^{*}(\hat{y}; \beta_1) - b] + \hat{y}. 
\end{equation}
The mapping $G(\cdot; \beta_1)$ is called gradient mapping of the function $d(\cdot; \beta_1)$ (see \cite{Nesterov2004}).

\subsection{A decomposition scheme with primal-dual update}
First, we adapt the scheme \eqref{eq:main_update_rule}-\eqref{eq:main_update_rule_beta} in the framework of primal and dual variant.
Suppose that the pair $(\bar{x}, \bar{y})\in X\times\mathbb{R}^m$ satisfies the excessive gap condition \eqref{eq:excessive_gap}. The primal step is computed as
follows:
\begin{equation}\label{eq:primal_step}
(\bar{x}^{+}, \bar{y}^{+}) := \mathcal{A}^p(\bar{x}, \bar{y}; \beta_1, \beta_2, \tau) ~\Longleftrightarrow~
\begin{cases}
\hat{x} := (1-\tau)\bar{x} + \tau x^{*}(\bar{y}; \beta_1),\\
\bar{y}^{+} := (1-\tau)\bar{y} + \tau y^{*}(\hat{x}; \beta_2), \\
\bar{x}^{+} := P(\hat{x}; \beta_2), 
\end{cases} 
\end{equation}
and then we update $\beta_1^{+} := (1-\tau)\beta_1$, where $\tau \in (0, 1)$ and $P(\cdot; \beta_2)$ is defined in \eqref{eq:prox_mapping_xy}. 
The difference between schemes $\mathcal{A}^p_m$ and $\mathcal{A}^p$ is that the parameter $\beta_2$ is fixed in $\mathcal{A}^p$.

Symmetrically, the dual step is computed as:
\begin{align}\label{eq:dual_step}
&(\bar{x}^{+}, \bar{y}^{+}) := \mathcal{A}^d(\bar{x},\bar{y}; \beta_1, \beta_2, \tau) \Longleftrightarrow 
\begin{cases}
&\hat{y} := (1-\tau)\bar{y} + \tau y^{*}(\bar{x}; \beta_2),\\
&\bar{x}^{+} : = (1-\tau)\bar{x} + \tau x^{*}(\hat{y}; \beta_1),\\
&\bar{y}^{+} := G(\hat{y}; \beta_1),
\end{cases} 
\end{align}
where $\tau \in (0,1)$. The parameter $\beta_1$ is kept unchanged, while $\beta_2$ is updated by $\beta_2^{+} := (1-\tau)\beta_2$.

The following result shows that $(\bar{x}^{+}, \bar{y}^{+})$ generated either by $\mathcal{A}^p$ or by $\mathcal{A}^d$ maintains the excessive gap condition \eqref{eq:excessive_gap}.

\begin{lemma}\label{le:dual_main}
Suppose that $(\bar{x}, \bar{y})\in X\times\mathbb{R}^m$ satisfy \eqref{eq:excessive_gap} with respect to two values $\beta_1$ and $\beta_2$. Then
$(\bar{x}^{+}, \bar{y}^{+})$ generated either by scheme $\mathcal{A}^p$ or by  $\mathcal{A}^d$ is in $X\times\mathbb{R}^m$ and maintains the excessive gap
condition \eqref{eq:excessive_gap} with respect to either two new values $\beta_1^{+}$ and $\beta_2$ or $\beta_1$ and $\beta_2^{+}$ provided that the
following condition holds:
\begin{equation}\label{eq:main_condition_new}
\beta_1\beta_2 \geq \frac{2\tau^2}{1-\tau}\max_{1\leq i\leq 2}\left\{\frac{\norm{A_i}^2}{\sigma_i}\right\}. 
\end{equation}
\end{lemma}
The proof of this lemma is quite similar to \cite[Theorem 4.2.]{Nesterov2005} that we omit here.

\begin{remark}\label{re:dual_init}
Given $\beta_1 > 0$, we can choose $\beta_2 > 0$ such that the condition \eqref{eq:initial_point_cond} holds. 
Let $y_c := 0 \in\mathbb{R}^m$, we compute a point $(\bar{x}^0, \bar{y}^0)$ as:
\begin{equation}\label{eq:init_dual}
\bar{x}^0 := x^{*}(y^c;  \beta_1) ~~ \textrm{and} ~~ \bar{y}^0 := G(y^c; \beta_1) = \frac{1}{L_{d}(\beta_1)}(A\bar{x} - c) + y^c.  
\end{equation}
Then, similar to \eqref{eq:initial_point}, the point $(\bar{x}^0, \bar{y}^0)$ satisfies \eqref{eq:excessive_gap}. Therefore, we can use this point as a starting
point for Algorithm \ref{alg:A2} below. 
\end{remark}

In Algorithm \ref{alg:A2} below we apply either the primal scheme $\mathcal{A}^p$ or the dual scheme $\mathcal{A}^d$ by using the following rule:

\noindent\textbf{Rule A. }\textit{If the iteration counter $k$ is even then apply $\mathcal{A}^p$. Otherwise, $\mathcal{A}^d$ is used.}

Now, we provide an update rule to generate a sequence $\{\tau_k\}$ such that the condition \eqref{eq:main_condition_new} holds.
Let $\bar{L} := 2\displaystyle\max_{1\leq i\leq 2}\left\{\frac{\norm{A_{i}}^2}{\sigma_i}\right\}$. 
Suppose that at the iteration $k$ the condition \eqref{eq:main_condition_new} holds, i.e.:
\begin{equation}\label{eq:cond_1}
\beta_1^k\beta_2^k \geq \frac{\tau_k^2}{1-\tau_k}\bar{L}. 
\end{equation}
Since at the iteration $k+1$, we either update $\beta_1^k$ or $\beta_2^k$. Thus we have $\beta_1^{k+1}\beta_2^{k+1} = (1-\tau_k)\beta_1^k\beta_2^k$. However,
as the condition \eqref{eq:cond_1} holds, we have $(1-\tau_k)\beta_1^k\beta_2^k \geq
\tau_k^2\bar{L}$.
Now, we suppose that the condition \eqref{eq:main_condition_new} is satisfied with $\beta_1^{k+1}$ and $\beta_2^{k+1}$, i.e.:
\begin{equation}\label{eq:cond_2}
\beta_1^{k+1}\beta_2^{k+1} \geq \frac{\tau_{k+1}^2}{1-\tau_{k+1}}\bar{L}. 
\end{equation}
This condition holds if $\tau_k^2\bar{L} \geq \frac{\tau_{k+1}^2}{1-\tau_{k+1}}\bar{L}$, which leads to $\tau^2_{k+1} + \tau_k^2\tau_{k+1} - \tau_k^2 \leq 0$.
Since $\tau_k, \tau_{k+1} \in (0,1)$, we obtain:
\begin{equation}\label{eq:tau_k_update}
0 < \tau_{k+1} \leq \frac{\tau_k}{2}\left[\sqrt{\tau^2_k + 4}-\tau_k\right] < \tau_k. 
\end{equation}
The tightest rule for updating $\tau_k$ is: 
\begin{equation}\label{eq:tau_update}
\tau_{k+1} := \frac{\tau_k}{2}\left[\sqrt{\tau^2_k + 4}-\tau_k\right], 
\end{equation}
for all $k\geq 0$ and $\tau_0\in (0,1)$ given. 
Associated with $\{\tau_k\}$, we generate two sequences $\{\beta_1^k\}$ and $\{\beta_2^k\}$ as:
\begin{equation}\label{eq:b1b2_update}
\beta_1^{k+1} := \begin{cases}(1-\tau_k)\beta_1^k ~&\textrm{if $k$ is even}\\ \beta_1^k &\textrm{otherwise},\end{cases} ~~\textrm{and}~~
\beta_2^{k+1} := \begin{cases}\beta_2^k ~&\textrm{if $k$ is even}\\ (1-\tau_k)\beta_2^k &\textrm{otherwise},\end{cases} 
\end{equation}
where $\beta_1^0 = \beta_2^0 = \bar{\beta} > 0$ are fixed.

\begin{lemma}\label{le:update_tau}
Let $\{\tau_k\}$, $\{\beta_1^k\}$ and $\{\beta_2^k\}$ be three sequences generated by \eqref{eq:tau_update} and \eqref{eq:b1b2_update}, respectively. Then:
\begin{equation}\label{eq:tau_est2}
\frac{(1-\tau_0)\bar{\beta}}{2\tau_0k+1} < \beta_1^k < \frac{2\bar{\beta}\sqrt{1-\tau_0}}{\tau_0k}, ~~\mathrm{and}~~
\frac{\bar{\beta}\sqrt{1-\tau_0}}{2\tau_0k+1} < \beta_2^k < \frac{2\bar{\beta}}{\tau_0k},
\end{equation}
for all $k\geq 1$.
\end{lemma}
The proof of this lemma can be found in the appendix.

\begin{remark}\label{re:choice_of_tau2}
We can see that the right-hand side $\eta_k(\tau_0) := \frac{4\bar{\beta}\sqrt{1-\tau_0}}{\tau_0(k+\tau_0)}$ of \eqref{eq:tau_est2} is decreasing in $(0,1)$ for
$k\geq 1$. Therefore, we can choose $\tau_0$ as large as possible to minimize $\eta_k(\cdot)$ in $(0, 1)$. For instance, we can choose $\tau_0 := 0.998$ in
Algorithm \ref{alg:A2}.
\end{remark}
Note that Lemma \ref{le:update_tau} shows that $\tau_k \sim O(\frac{1}{k})$. Hence, in Algorithm \ref{alg:A2}, we can also use a
simple updating rule for $\tau_k$ as $\tau_k = \frac{a}{k+b}$, where $a \in (\frac{3}{2}, 2)$ and $b\geq
\frac{a-1}{2-a} > 0$. This update satisfies \eqref{eq:main_condition_new}.

\subsection{The algorithm and its worst-case complexity}
Suppose that the initial point $(\bar{x}^0, \bar{y}^0)$ is computed by \eqref{eq:init_dual}. Then, we can choose $\beta_1^0 = \beta_2^0 =
\sqrt{2\displaystyle\max_{1\leq i\leq 2}\left\{\frac{\norm{A_{i}}^2}{\sigma_i} \right\}}$ which satisfy \eqref{eq:initial_point_cond}. 
The algorithm is now presented in detail as follows:

\vskip0.3cm
\noindent\rule[1pt]{\textwidth}{1.0pt}{~~}
\begin{algorithm}\vskip-0.2cm\label{alg:A2}{~}(\textit{Decomposition Algorithm with Primal-Dual Update})
\end{algorithm}
\vskip -0.3cm
\noindent\rule[1pt]{\textwidth}{0.5pt}
\noindent\textbf{Initialization: }
\begin{enumerate}
\item Choose $\tau_0 := 0.998$ and set $\beta_{1}^0 = \beta_2^0 := \sqrt{2\max_{1\leq i\leq 2}\left\{ \frac{\norm{A_{i}}^2}{\sigma_i}\right\}}$.  
\item Compute $\bar{x}^0$ and $\bar{y}^0$ as:
\begin{equation*}
\bar{x}^0 := x^{*}(y^c; \beta_1^0), ~\textrm{and}~ \bar{y}^0 :=  \frac{1}{L_{d}(\beta_1^0)}(A\bar{x}^0 - b) + y^c. 
\end{equation*}
\end{enumerate}
\noindent\textbf{Iteration: }\texttt{For} $k=0,1,\cdots$ \texttt{do}
\begin{enumerate}
\item If a given stopping criterion is satisfied then terminate.
\item If $k$ is even then:
\begin{itemize}
\item[]\textrm{2a)} Compute $(\bar{x}^{k+1}, \bar{y}^{k+1})$ as:
\begin{equation*}\label{eq:main_step_primal}
(\bar{x}^{k+1}, \bar{y}^{k+1}) := \mathcal{A}^p(\bar{x}^k, \bar{y}^k; \beta_1^k, \beta_2^k, \tau_k). 
\end{equation*}
\item[]\textrm{2b)} Update the smoothness parameter $\beta_1^k$ as $\beta_1^{k+1} := (1-\tau_k)\beta_1^k$.
\end{itemize}
\item Otherwise, i.e. if $k$ is odd then:
\begin{itemize}
\item[]\textrm{3a)} Compute $(\bar{x}^{k+1}, \bar{y}^{k+1})$ as:
\begin{equation*}\label{eq:main_step_dual}
(\bar{x}^{k+1}, \bar{y}^{k+1}) := \mathcal{A}^d(\bar{x}^k, \bar{y}^k; \beta_1^k, \beta_2^k, \tau_k). 
\end{equation*}
\item[]\textrm{3b)} Update the smoothness parameter $\beta_2^k$ as $\beta_2^{k+1} := (1-\tau_k)\beta_2^k$.
\end{itemize}
\item Update the step size parameter $\tau_k$ as: $\tau_{k+1} := \frac{\tau_k}{2}\left[\sqrt{\tau^2_k+4}-\tau_k\right]$. 
\end{enumerate}
\texttt{End of For}.
\vskip-0.2cm
\noindent\rule[1pt]{\textwidth}{1.0pt}

The main steps of Algorithm \ref{alg:A2} are Steps 2a and 2b, which requires us to compute either a primal step or a dual step. In the primal step,
we need to solve two convex problem pairs in parallel, while in the dual step, it only requires to solve two convex problems in parallel. 
The following theorem shows the convergence of this algorithm.

\begin{theorem}\label{th:convergence2}
Let the sequence $\{ (\bar{x}^k,\bar{y}^k) \}_{k\geq 0}$ be generated by Algorithm \ref{alg:A2}. Then the duality and feasibility gaps satisfy:
\begin{eqnarray}
&\phi(\bar{x}^k) - d(\bar{y}^k) &\leq \frac{2\sqrt{\bar{L}}(D_1 + D_2)}{0.998k}, \label{eq:duality_gap2}\\
[-1.5ex]
\textrm{and}~ && \nonumber\\
[-1.5ex]
& \norm{A\bar{x}^k - b} &\leq \frac{2\sqrt{\bar{L}}}{0.998k}\left[\norm{y^{*}} + \sqrt{\norm{y^{*}}^2 + 2(D_1 +
D_2)}\right]\label{eq:feasible_gap2},
\end{eqnarray}
where $\bar{L} := 2\displaystyle\max_{1\leq i\leq 2}\left\{ \frac{\norm{A_{i}}^2}{\sigma_i}\right\}$ and $k\geq 1$.
\end{theorem}

\begin{proof}
The conclusion of this theorem follows directly from Lemmas \ref{le:excessive_gap} and \ref{le:choice_of_tau}, the condition $\tau_0 = 0.998$, $\beta^0_1 =
\beta_2^0 = \sqrt{\bar{L}}$ and the fact that $\beta_1^k\leq
\beta_2^k$. 
\Eproof
\end{proof}

\begin{remark}\label{re:complexity2}
Note that the worst-case complexity of Algorithm \ref{alg:A2} is still $O(\frac{1}{\varepsilon})$. The constants in the complexity estimates
\eqref{eq:duality_gap} and \eqref{eq:feasible_gap} are similar to the one in \eqref{eq:duality_gap2} and \eqref{eq:feasible_gap2}, respectively.
As we discuss in Section \ref{sec:discussion} below, the rate of decrease of $\tau_k$ in Algorithm \ref{alg:A2} is smaller than two times of $\tau_k$ in Algorithm \ref{alg:A1}. Consequently, the sequences $\{\beta_1^k\}$ and $\{\beta_2^k\}$ generated by Algorithm \ref{alg:A1}
approach zero faster than the ones generated by Algorithm \ref{alg:A2}.
\end{remark}
 
\begin{remark}\label{re:B_rule}
Note that the role of the schemes $\mathcal{A}^p$ and $\mathcal{A}^d$ in Algorithm \ref{alg:A2} can be exchanged. Therefore,  Algorithm \ref{alg:A2} can be
modified at three steps to obtain a symmetric variant as follows:
\begin{enumerate}
\item At Step 2 of the initialization phase, \eqref{eq:initial_point} to compute $\bar{x}^0$ and $\bar{y}^0$ instead of \eqref{eq:init_dual}. 
\item At Steps 2a, $\mathcal{A}^p$ is used if the iteration counter $k$ is odd. Otherwise, we use $\mathcal{A}^d$ at Step 3a.
\item At Steps 2b, $\beta_2^k$ is updated if $k$ is odd. Otherwise, $\beta_1^k$ is updated at Step 3b.
\end{enumerate}
\end{remark}

\section{Application to strongly convex programming problems}\label{sec:strongly_convex_case}
If $\phi_{i}$ ($i=1,2$) in \eqref{eq:distNLP} is strongly convex then the convergence rate of the dual scheme \eqref{eq:dual_step} can be accelerated up to $O(\frac{1}{k^2})$. 

Suppose that $\phi_i$ is strongly convex with a convexity parameters $\sigma_i > 0$ ($i=1,2$). Then the function $d$ defined by \eqref{eq:dual_func} is
well-defined,
concave and differentiable. Moreover, its gradient is given by:
\begin{equation}\label{eq:sec5_dphi}
\nabla d(y) = A_{1}x^{*}_{1}(y) + A_{2}x^{*}_{2}(y) - b, 
\end{equation}
which is Lipschitz continuous with a Lipschitz constant $L^{\phi} := \frac{\norm{A_{1}}^2}{\sigma_1} + \frac{\norm{A_{2}}^2}{\sigma_2}$.
The excessive gap condition \eqref{eq:excessive_gap} in this case becomes:
\begin{equation}\label{eq:sec5_excessive_gap}
f(\bar{x}; \beta_2) \leq d(\bar{y}), 
\end{equation}
for given $\bar{x} \in X$, $\bar{y}\in\mathbb{R}^m$ and $\beta_2 > 0$.
From Lemma \ref{le:excessive_gap} we conclude that if the point $(\bar{x}, \bar{y})$ satisfies \eqref{eq:sec5_excessive_gap} then, for a given $y^{*}\in Y^{*}$, the following
estimates hold:
\begin{eqnarray}
& -2\beta_2\norm{y^{*}}^2 \leq -\norm{y^{*}}\norm{A\bar{x} - b} &\leq \phi(\bar{x}) - d(\bar{y}) \leq 0, \label{eq:sec5_dual_gap}\\
[-1.0ex]
\mathrm{and} && \nonumber\\
[-1.0ex]
& \norm{A\bar{x} - b} &\leq 2\beta_2\norm{y^{*}}\label{eq:sec5_feasible_gap}. 
\end{eqnarray}
We now adapt the dual scheme \eqref{eq:dual_step} to this special case. 
Suppose $(\bar{x},\bar{y})\in X\times\mathbb{R}^m$ satisfies \eqref{eq:sec5_excessive_gap}, we generate a new pair $(\bar{x}^{+},\bar{y}^{+})$ as
\begin{equation}\label{eq:sec5_dual_step}
(\bar{x}^{+},\bar{y}^{+}) := \mathcal{A}^d_s(\bar{x},\bar{y}; \beta_2, \tau) \Longleftrightarrow
\begin{cases}
&\hat{y} := (1-\tau)\bar{y} + \tau y^{*}(\bar{x}; \beta_2),\\
&\bar{x}^{+} := (1-\tau)\bar{x} + \tau x^{*}(\hat{y}),\\
&\bar{y}^{+} = \frac{1}{L^{\phi}}(Ax^{*}(\hat{y}) - b) + \hat{y},
\end{cases}
\end{equation}
where $y^{*}(\bar{x}; \beta_2) = \frac{1}{\beta_2}(A\bar{x} - b)$, and $x^{*}(y) := (x^{*}_{1}(y), x_{2}^{*}(y))$ is the solution of the minimization problem in
\eqref{eq:dual_func}. The parameter $\beta_2$ is updated by $\beta^{+}_2 := (1-\tau)\beta_2$ and $\tau \in (0, 1)$ will appropriately be chosen.

The following lemma shows that $(\bar{x}^{+}, \bar{y}^{+})$ generated by \eqref{eq:sec5_dual_step} satisfies \eqref{eq:sec5_excessive_gap} whose proof can be found in \cite{Nesterov2005}.

\begin{lemma}\label{le:sec5_main_rule}
Suppose that the point $(\bar{x},\bar{y})\in X\times\mathbb{R}^m$ satisfies the excessive gap condition \eqref{eq:sec5_excessive_gap} with the value $\beta_2$.
Then the new point $(\bar{x}^{+},\bar{y}^{+})$ computed by \eqref{eq:sec5_dual_step} is in $X\times\mathbb{R}^m$ and also satisfies
\eqref{eq:sec5_excessive_gap} with a new parameter value $\beta_2^{+}$ provided that 
\begin{equation}\label{eq:sec5_main_cond}
\beta_2 \geq \frac{\tau^2L_{\phi}}{1-\tau}. 
\end{equation}
\end{lemma}
Now, let us derive the rule to update the parameter $\tau$. Suppose that $\beta_2$ satisfies \eqref{eq:sec5_main_cond}. Since $\beta^{+}_2 = (1-\tau)\beta_2$, the condition \eqref{eq:sec5_main_cond}
holds for $\beta_2^{+}$ if $\tau^2 \geq \frac{\tau_{+}^2}{1-\tau_{+}}$. Therefore, similar to Algorithm \ref{alg:A2}, we update the parameter $\tau$ by using the rule \eqref{eq:tau_sequence}. The conclusion
of Lemma \ref{le:update_tau} still holds for this case. 

Before presenting the algorithm, it is necessary to find a starting point $(\bar{x}^0, \bar{y}^0)$ which satisfies \eqref{eq:sec5_excessive_gap}. Let $y^c = 0\in\mathbb{R}^m$ and $\beta_2 = L^{\phi}$. We compute $(\bar{x}^0$, $\bar{y}^0)$ as
\begin{equation}\label{eq:sec5_init_point}
\bar{x}^0 := x^{*}(y^c) ~~\textrm{and}~~ \bar{y}^0 := \frac{1}{L^{\phi}}(A\bar{x}^0 - b) + y^c.
\end{equation}
It follows from Lemma 7.4 \cite{Nesterov2005} that $(\bar{x}^0, \bar{y}^0)$ satisfies the excessive gap condition \eqref{eq:sec5_excessive_gap}.

Finally, the decomposition algorithm for solving the strongly convex programming problem of the form \eqref{eq:distNLP} is described in detail as follows:

\vskip0.3cm
\noindent\rule[1pt]{\textwidth}{1.0pt}{~~}
\begin{algorithm}\vskip -0.2cm\label{alg:A3}{~}
(\textit{Decomposition algorithm for strongly convex objective function})
\end{algorithm}
\vskip -0.3cm
\noindent\rule[1pt]{\textwidth}{0.5pt}
\noindent\textbf{Initialization: }
\begin{enumerate}
\item Choose $\tau_0 := 0.5$. Set $\beta_2^0 = \frac{\norm{A_{1}}^2}{\sigma_1} + \frac{\norm{A_{2}}^2}{\sigma_2}$.
\item Compute $\bar{x}^0$ and $\bar{y}^0$ as:
\begin{equation*}
\bar{x}^0 := x^{*}(y^c) ~\textrm{and}~ \bar{y}^0 := \frac{1}{L^{\phi}}(A\bar{x}^0 - b) + y^c.
\end{equation*}
\end{enumerate}
\noindent\textbf{Iteration: }\texttt{For} $k=0,1,\cdots$ \texttt{do}
\begin{enumerate}
\item If a given stopping criterion is satisfied then terminate.
\item Compute $(\bar{x}^{k+1},\bar{y}^{k+1})$ using scheme \eqref{eq:sec5_dual_step}:
\begin{equation*}\label{eq:sec5_main_step}
(\bar{x}^{k+1}, \bar{y}^{k+1}) := \mathcal{A}^d_s(\bar{x}^k, \bar{y}^k; \beta_2^k, \tau_k). 
\end{equation*}
\item Update the smoothness parameter as: $\beta_2^{k+1} := (1-\tau_k)\beta_2^k$. 
\item Update the step size parameter $\tau_k$ as: $\tau_{k+1} := \frac{\tau_k}{2}\left[\sqrt{\tau^2_k+4}-\tau_k\right]$. 
\end{enumerate}
\texttt{End of For}.
\vskip-0.2cm
\noindent\rule[1pt]{\textwidth}{1.0pt}
The convergence and the worst-case complexity of Algorithm \ref{alg:A3} are stated as in Theorem \ref{th:convergence3} below.

\begin{theorem}\label{th:convergence3}
Let $\{(\bar{x}^k,\bar{y}^k)\}_{k\geq 0}$ be a sequence generated by Algorithm \ref{alg:A3}. Then the following duality and feasibility gaps are satisfied:
\begin{eqnarray}
& -\frac{8L^{\phi}\norm{y^{*}}^2}{(k+4)^2} &\leq \phi(\bar{x}^k) - d(\bar{y}^k) \leq 0, \label{eq:sec5_dual_gap2}\\
[-1.5ex]
\textrm{and}~ && \nonumber\\
[-1.5ex]
& \norm{A\bar{x}^k - b} &\leq \frac{8L^{\phi}\norm{y^{*}}}{(k+4)^2}, \label{eq:sec5_feasible_gap2}
\end{eqnarray}
where $L^{\phi} := \frac{\norm{A_{1}}^2}{\sigma_1} + \frac{\norm{A_{2}}^2}{\sigma_2}$.
\end{theorem}

\begin{proof}
From the update rule of $\tau^k$, we have $(1-\tau_{k+1}) = \frac{\tau_{k+1}^2}{\tau_k^2}$. Moreover, since $\beta_2^{k+1} = (1-\tau_k)\beta_2^k$, it implies that $\beta_2^{k+1} =
\beta_2^0\prod_{i=0}^{k}(1-\tau_i) = \frac{\beta_2^0(1-\tau_0)}{\tau_0^2}\tau_k^2$. By using the inequalities \eqref{eq:lm42_est1} and $\beta_2^0 = L_{\phi}$,
we have $\beta_2^{k+1} < \frac{4L_{\phi}(1-\tau_0)}{(\tau_0k + 2)^2}$. With $\tau_0 = 0.5$, one has $\beta_2^k < \frac{8L_{\phi}}{(k+4)^2}$. By substituting
this inequality into \eqref{eq:sec5_dual_gap} and \eqref{eq:sec5_feasible_gap}, we obtain \eqref{eq:sec5_dual_gap2} and \eqref{eq:sec5_feasible_gap2},
respectively.
\Eproof
\end{proof}
Theorem \ref{th:convergence3} shows that the worst-case complexity of Algorithm \ref{alg:A3} is $O(\frac{1}{\sqrt{\varepsilon}})$. Moreover, at each iteration
of this algorithm, only two convex problems need to be solved \textit{in parallel}.
 
\section{Discussion on implementation and comparison}\label{sec:discussion}
\subsection{The choice of prox-functions and the Bregman distance}
Algorithms \ref{alg:A1} and \ref{alg:A2} require to build a prox-function for each feasible set $X_{i}$ for $i=1,2$. 
For a nonempty, closed and bounded convex set $X_{i}$, the simplest prox-function is $p_i(x_{i}) := \frac{\rho_i}{2}\norm{x_{i} -\bar{x}_{i}}^2$, for a given
$\bar{x}_{i}\in X_{i}$ and $\rho_i > 0$. This function is strongly convex with the parameter $\sigma_i = \rho_i$ and the prox-center is $\bar{x}_{i}$,
($i=1,2$). 
In implementation, it is worth to investigate the structure of the feasible set $X_{i}$ in order to choose an appropriate prox-function and its scaling factor
$\rho_i$ for each feasible subset $X_{i}$ ($i=1,2$).

In \eqref{eq:prox_mapping_xy}, we have used the Euclidean distance to construct the proximal terms. It is possible to use a generalized Bregman distance in
these problems which is compatible to the prox-function $p_i$ and the feasible subset $X_i$ ($i=1,2$).
Moreover, a proper choice of the norms in the implementation may lead to a better performance of the algorithms, see \cite{Nesterov2005} for more details.

\subsection{Extension to a multi-component separable objective function}
The algorithms developed in the previous sections can be directly applied to solve problem \eqref{eq:separable_convex_problem} in the case $M > 2$.
First, we provide the following formulas to compute the parameters of Algorithms \ref{alg:A1}-\ref{alg:A3}.
\begin{enumerate}
\item The constant $\bar{L}$ in Theorems \ref{th:convergence} and \ref{th:convergence2} is replaced by $\bar{L}_M = M\displaystyle\max_{1\leq i\leq M}\left\{\frac{\norm{A_{i}}^2}{\sigma_i}\right\}$.
\item The initial values of $\beta_1^0$ and $\beta_2^0$ in Algorithms \ref{alg:A2} and \ref{alg:A3} are $\beta_1^0 = \beta_2^0 = \sqrt{\bar{L}_M}$.
\item The Lipschitz constant $L_i^{\psi}(\beta_2)$ in Lemma \ref{le:Lipschitz_diff_psi} is $L_i^{\psi}(\beta_2) = \frac{M\norm{A_{i}}^2}{\beta_2}$ ($i=1,\dots, M$).
\item The Lipschitz constant $L_{d}(\beta_1)$ in Lemma \ref{le:smoothing_estimate} is $L_{d}(\beta_1) := \frac{1}{\beta_1}\displaystyle\sum_{i=1}^M\frac{\norm{A_{i}}^2}{\sigma_{i}}$.
\item The Lipschitz constant $L_{\phi}$ in Algorithm \ref{alg:A3} is $L^{\phi} := \displaystyle\sum_{i=1}^M\frac{\norm{A_{i}}^2}{\sigma_{i}}$. 
\end{enumerate}
Note that these constants depend linearly on $M$ and the structure of matrix $A_{i}$ ($i=1,\dots, M$). 

Next, we rewrite the smoothed dual function $d(y;\beta_1)$ defined by \eqref{eq:d_lbd} for the case $M > 2$ as follows:
\begin{eqnarray*}
d(y;\beta_1) = \sum_{i=1}^Md_i(y; \beta_1),
\end{eqnarray*}
where $M$ function values $d_i(y;\beta_1)$ can be computed in parallel as:
\begin{eqnarray*}
d_i(y;\beta_1) = -\frac{1}{M}b_i^Ty + \min_{x_i\in X_i}\left\{ \phi_i(x_i) + y^TA_ix_i + \beta_1p_i(x_i)\right\}.
\end{eqnarray*}
Note that the term $-\frac{1}{M}b_i^Ty$ is also computed locally for each component subproblem instead of computing separately as in \eqref{eq:d_lbd}.
The quantities $\hat{y}$ and $y^{+} := G(\hat{y};\beta_1)$ defined in \eqref{eq:primal_step} and \eqref{eq:dual_step} can respectively be expressed as:
\begin{align}
& \hat{y} := (1-\tau)\bar{y} + (1-\tau)\sum_{i=1}^M\frac{1}{\beta_2}(A_i\bar{x}_i - \frac{1}{M}b), \nonumber\\
\textrm{and}~ &y^{+} := \hat{y} + \sum_{i=1}^M\left[\frac{1}{L^d(\beta_1)}(A_ix_i^{*}(\hat{y};\beta_1) - \frac{1}{M}b) \right].\nonumber
\end{align}
These formulas show that each component of $\hat{y}$ and $y^{+}$ can be computed by only using the local information and its neighborhood information.
Therefore, both algorithms are highly distributed.

Finally, we note that if there exists a component $\phi_i$ of the objective function $\phi$ which is Lipschitz continuously differentiable then the gradient
projection mapping $G_i(\hat{x};\beta_2)$ defined by \eqref{eq:G_XY} corresponding to the primal convex subproblem of this component can be used instead of the 
proximity mapping $P_i(\hat{x};\beta_2)$ defined by \eqref{eq:prox_mapping_xy}. This modification can reduce the computational cost of the
algorithms.
Note that the sequence $\{\tau_k\}_{k\geq 0}$ generated by the rule \eqref{eq:tau_sequence} still maintains the
condition \eqref{eq:main_condition_diff} in Remark \ref{re:Lipschit_diff_phi}.

\subsection{Stopping criterion}
In practice, we do not often encounter a problem which reaches the worst-case complexity bound. Therefore, it is necessary to provide a stopping criterion for the implementation of Algorithms
\ref{alg:A1}, \ref{alg:A2} and \ref{alg:A3} to terminate earlier than using the worst-case bound.
In principle, we can use the KKT condition to terminate the algorithms. However, evaluating the global KKT tolerance in a distributed manner is
impractical.

From Theorems \ref{th:convergence} and \ref{th:convergence2} we see that the upper bound of the duality and feasibility gaps do not only depend on the
iteration counter $k$
but also on the constants $\bar{L}$, $D_i$ and $y^{*}\in Y^{*}$. The constant $\bar{L}$ can be explicitly computed based on matrix $A$ and the choice of the
prox-functions. 
We now discuss on the evaluations of $D_i$ and $y^{*}$ in the case $X_i$ is unbounded.
Let sequence $\{(\bar{x}^k,\bar{y}^k)\}$ be generated by Algorithm \ref{alg:A1} (or Algorithm \ref{alg:A2}). Suppose that $\{(\bar{x}^k, \bar{y}^k)\}$ converges
to $(x^{*}, y^{*})\in X^{*}\times Y^{*}$. Thus, for $k$ sufficiently large, the sequence $\{(\bar{x}^k,\bar{y}^k)\}$ is contained in a neighborhood of
$X^{*}\times Y^{*}$. 
Given $\omega > 0$, let us define
\begin{equation}\label{eq:hat_DXDYDL}
\hat{D}_i^k := \max_{0\leq j\leq k}p_i(\bar{x}^j_{i}) + \omega ~\textrm{and}~ \hat{y}^k := \max_{0\leq j\leq k}\norm{\bar{y}^j} +
\omega.
\end{equation}
We can use these constants to construct a stopping criterion in Algorithms \ref{alg:A1} and \ref{alg:A2}. 
More precisely, for a given tolerance $\varepsilon > 0$, we compute 
\begin{equation}\label{eq:terminate_alg}
e_d := \beta_1^k(\hat{D}^k_1+\hat{D}^k_2), ~\textrm{and}~ 
e_p := \beta_2^k\left[\hat{y}^k+\sqrt{(\hat{y}^k)^2 + 2(\hat{D}^k_1+\hat{D}_2^k)}\right],  
\end{equation}
at each iteration. We terminate Algorithm \ref{alg:A1} if $e_d \leq \varepsilon$ and $e_p \leq \varepsilon$. A similar strategy can also be applied to Algorithms \ref{alg:A2} and \ref{alg:A3}.

\subsection{Comparison.}
Firstly, we compare Algorithms \ref{alg:A1} and \ref{alg:A2}.
From Lemma \ref{le:excessive_gap} and the proof of Theorems \ref{th:convergence} and \ref{th:convergence2} we see that the rate of convergence of both
algorithms is as same as of $\beta_1^k$ and $\beta_2^k$.
At each iteration, Algorithm \ref{alg:A1} updates simultaneously $\beta_1^k$ and $\beta_2^k$ by using the same value of $\tau_k$, while Algorithm \ref{alg:A2}
updates only one parameter.
Therefore, to update both parameters $\beta_1^k$ and $\beta_2^k$, Algorithm \ref{alg:A2} needs two iterations. 
We analyze the update rule of $\tau_k$ in Algorithms \ref{alg:A1} and \ref{alg:A2} to compare the rate of convergence of both algorithms.
 
Let us define
\begin{equation*}
\xi_1(\tau) := \frac{\tau}{\tau + 1} ~\mathrm{and}~ \xi_2(\tau) := \frac{\tau}{2}\left[\sqrt{\tau^2 + 4} - \tau\right].
\end{equation*}
The function $\xi_2$ can be rewritten as $\xi_2(\tau) = \frac{\tau}{\sqrt{(\tau/2)^2 + 1} + \tau/2}$.
Therefore, we can easily show that:
\begin{equation*}
\xi_1(\tau) < \xi_2(\tau) < 2\xi_1(\tau).
\end{equation*}
If we denote by $\{\tau_k^{\mathrm{A}_1}\}_{k\geq 0}$ and $\{\tau_k^{\mathrm{A}_2}\}_{k\geq 0}$ the two sequences generated by Algorithms \ref{alg:A1} and
\ref{alg:A2}, respectively then we have $\tau_k^{\mathrm{A}_1} < \tau_k^{\mathrm{A}_2} < 2\tau_k^{\mathrm{A}_1}$ for all $k$ provided that
$2\tau_0^{\mathrm{A}_1} \geq \tau_0^{\mathrm{A}_2}$.
Since Algorithm \ref{alg:A1} updates $\beta_1^{k}$ and $\beta_2^k$ simultaneously while Algorithm \ref{alg:A2} updates each of them at each iteration. 
If we choose $\tau_0^{\mathrm{A}_1} = 0.499$ and $\tau_0^{\mathrm{A}_2} = 0.998$ in Algorithms \ref{alg:A1} and \ref{alg:A2}, respectively, then, by directly
computing the value of $\tau^{\mathrm{A}_1}_k$ and $\tau^{\mathrm{A}_2}_k$, we can see that $2\tau^{\mathrm{A}_1}_k > 2\tau^{\mathrm{A}_2}_k$ for all $k\geq 1$.
Consequently, the sequences $\{\beta_1^k\}$ and $\{\beta_2^k\}$ in Algorithm \ref{alg:A1} converge to zero faster than in Algorithm \ref{alg:A2}. In other
words, Algorithm \ref{alg:A1} is faster than Algorithm \ref{alg:A2}. 

Now, we compare Algorithm \ref{alg:A1}, Algorithm \ref{alg:A2} and Algorithm 3.2. in \cite{Necoara2008} (see also \cite{Tsiaflakis2010}). 
Note that the smoothness parameter $\beta_1$ which is also denoted by $c$ is fixed in Algorithm 3.2 of \cite{Necoara2008}. Moreover, this parameter is
proportional to the given desired accuracy $\varepsilon$, which is often very small. Thus, the Lipschitz
constant $L^{d}(\beta_1)$ is very large. Consequently, Algorithm 3.2. of \cite{Necoara2008} makes a slow progress at the very early iterations. In Algorithms \ref{alg:A1} and \ref{alg:A2}, the parameters
$\beta_1$ and $\beta_2$ are dynamically updated starting from given values. 
Besides, the cost per iteration of Algorithm 3.2 \cite{Necoara2008} is more expensive than Algorithms \ref{alg:A1} and \ref{alg:A2} since it requires to solve two convex problem pairs in parallel and
two dual steps. 

\section{Numerical Tests}\label{sec:num_results}
In this section, we verify the performance of the proposed algorithms by applying them to solve the following separable convex optimization problem:
\begin{equation}\label{eq:scp_prob}
\left\{\begin{array}{cl}
\displaystyle\min_{x = (x_1,\dots, x_M)} &\Big\{ \phi(x) := \displaystyle\sum_{i=1}^M\phi_i(x_i) \Big\},\\
\textrm{s.t.} &\displaystyle\sum_{i=1}^Mx_i \leq (=) b,\\
& l_i \leq x_i \leq u_i, ~ i=0,\dots, M,
\end{array}\right.
\end{equation}
where $\phi_i : \mathrm{R}^{n_x}\to \mathrm{R}$ is convex, $b$, $l_i$ and $u_i\in\mathbb{R}^{n_x}$ are given for $i=1,\dots, M$.
The problem \eqref{eq:scp_prob} arises in many applications including resource allocation problems \cite{Ibaraki1980} and DSL dynamic spectrum management
problems \cite{Tsiaflakis2010}. 
In the case of inequality coupling constraints, we can bring the problem \eqref{eq:scp_prob} in to the form of  \eqref{eq:separable_convex_problem} by adding a
slack
variable $x_{M+1}$ as a new component.

\subsection{Implementation details}
We implement Algorithms \ref{alg:A1} and \ref{alg:A2} proposed in the previous sections to solve \eqref{eq:scp_prob}. The implementation is carried out in C++
running on a $16$ cores workstation Intel\textregistered Xeron $2.7$GHz and $12$ GB of RAM.
To solve general convex programming subproblems, we implement a primal-dual predictor-corrector interior point method. All the algorithms are parallelized by
using \texttt{OpenMP}.

The prox-functions $d_i(x_{i}) := \frac{\rho}{2}\norm{x_{i} - x_{i}^c}^2$ are used, where $x_{i}^c$ is the center of the box $X_{i} := [l_i, u_i]$ and $\rho :=
1$ for all $i=1,\dots, M$. 
We terminate Algorithms \ref{alg:A1} and \ref{alg:A2} if $\texttt{rpfgap} := \norm{Ax^k - b}_2/\norm{b}_2 \leq \varepsilon_{\mathrm{p}}$ and
either 
$\texttt{rdfgap} := \max\left\{ 0, \beta_1^k\sum_{i=1}^MD_{X_i} - \frac{1}{2\beta_2}\norm{Ax^k-b}^2 \right\} \leq \varepsilon_{\mathrm{d}}(\abs{\phi(x^k)} + 1)$
or the value of the objective function does not significantly change in $3$ successive iterations, i.e. $\abs{\phi(\bar{x}^{k}) -
\phi(\bar{x}^{k-j})}/\max\{1.0, \abs{\phi(\bar{x}^k)}\}
\leq \varepsilon_{\phi}$ for $j=1, 2, 3$, where $\varepsilon_{\mathrm{p}} = 10^{-2}$, $\varepsilon_{\mathrm{d}} = 10^{-1}$ and
$\varepsilon_{\phi} = 10^{-5}$ are given tolerances.
Note that the quantity \texttt{rdfgap} is computed in the worst-case complexity, see Lemma \ref{le:excessive_gap}.

To compare the performance of the algorithms, we also implement the proximal-center-based decomposition algorithm proposed in \cite[Algorithm
3.2.]{Necoara2008} and an exact variant of the proximal-based decomposition in \cite[Algorithm I]{Chen1994} for solving \eqref{eq:scp_prob} which we name
\texttt{PCBD} and
\texttt{EPBD},
respectively. The prox-function of the dual problem is chosen as $d_Y(y) := \frac{\rho}{2}\norm{y}^2$ with $\rho := 1.0$ and the smoothness parameter $c$ of
\texttt{PCBD} is
set to $c := \frac{\varepsilon_{\mathrm{p}}}{\sum_{i=1}^MD_{X_i}}$, where $D_{X_i}$ is defined by \eqref{eq:DxDy}. 
We terminate \texttt{PCBD} if the relative feasibility gap $\texttt{rpfgap}\leq\varepsilon_p$ and either the objective value reaches the one reported by
Algorithm \ref{alg:A1} or the maximum number of iterations $\texttt{maxiter} = 10,000$ is reached.
 
\subsection{Numerical results and comparison}
We test the above algorithms for three examples. The two first examples are resource allocation problems and the last one is a DSL dynamic spectrum management
problem. The first example was considered in \cite{Johansson2009}, while the problem formulation and the data of the third example are obtained from
\cite{Tsiaflakis2010}.

\vskip0.2cm
\noindent\textit{7.2.1. Resource allocation problems. }
Let us consider a resource allocation problem in the form of \eqref{eq:scp_prob} where the coupling constraint $\sum_{i=1}^Mx_i = b$ is tackled.

\vskip0.12cm
\noindent\textit{$\mathrm{(a)}$ Nonsmooth convex optimization problems. }
In the first numerical example, we choose $n_x = 1$, $M = 5$, the objective function $\phi_i(x_i) := i\abs{x_i - i}$ which is nonsmooth and $b = 10$ as in
\cite{Johansson2009}. The lower bound $l_i$ is set to $l_i=-5$ and the upper bound $u_i$ is $u_i = 7$ for $i=1,\dots, M$. With these choices,
the optimal solution of this problem is $x^{*} = (-4, 2, 3, 4, 5)$. 

We use four different algorithms which consist of Algorithm \ref{alg:A1}, Algorithm \ref{alg:A2}, \texttt{PCBD} in \cite{Necoara2008} and \texttt{PCBD} in
\cite[Algorithm I]{Chen1994} to solved problem \eqref{eq:scp_prob}. The approximate solutions reported by these algorithms after $100$ iterations are $x^k =
(-3.978, 2, 3, 4, 5)$, ~$(-3.875, 1.983,\\ 2.990, 3.996, 5)$, ~$(-4.055, 2, 3, 4, 5)$ and $(-4.423, 2, 3, 4, 5)$, respectively. The corresponding objective
values are $\phi(x^k) =
4.978$,~$4.954$, $5.055$ and $5.423$, respectively.

The convergence behaviour of four algorithms is shown in Figure \ref{fig:error_exam1}, where the relative error of the objective function
$\mathrm{re}_{\phi} := \abs{\phi(x^k)-\phi^{*}}/\abs{\phi^{*}}$ is plotted on the left and the relative error of the solution $\mathrm{re}_x :=
\norm{x^k-x^{*}}/\norm{x^{*}}$ is on the right.
\begin{figure}[ht]
\vskip-0.5cm
\centerline{\includegraphics[angle=0,height=3.5cm,width=12.0cm]{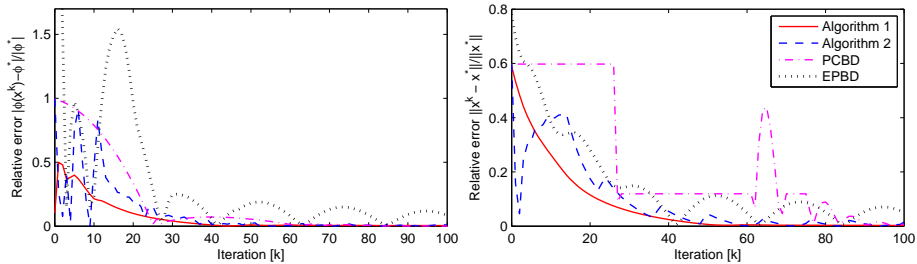}}
\caption{The relative error of the approximations to the optimal value (left) and to the optimal solution (right).}
\label{fig:error_exam1}
\vskip-0.4cm
\end{figure}
As we can see from these figures that the relative errors in Algorithm \ref{alg:A2}, \texttt{PCBD} and \texttt{EPBD} oscillate with respect to the
iteration counter while they are decreasing monotonously in Algorithm \ref{alg:A1}. 
The relative errors in Algorithms \ref{alg:A1} and \ref{alg:A2} are approaching zero earlier than the ones in \texttt{PCBD} and \texttt{EPBD}.
Note that in this example a nonmonotone variant of the \texttt{PCBD} algorithm \cite{Necoara2008,Tsiaflakis2010} is used.

\vskip0.12cm
\noindent\textit{$\mathrm{(b)}$ Nonlinear resource allocation problems.}
In order to compare the efficiency of Algorithm \ref{alg:A1}, Algorithm \ref{alg:A2} and \texttt{PCBD}, we build two performance profiles of these
algorithms in terms of total iterations and total computational time. 

In this case, the objective function $\phi_i$ is chosen as $\phi_i(x_i) = a_i^Tx_i - w_i\ln(1 + b_i^Tx_i)$, where the linear cost vector $a_i$, vector $b_i$
and the
weighting vector $w_i$ are generated randomly in the intervals $[0, 5]$, $[0, 10]$ and $[0, 5]$, respectively. The lower bound and the upper bound
are set to $l_i = (0, \dots, 0)^T$ and $u_i = (1,\dots, 1)^T$, respectively. Note that the objective function $\phi_i$ is linear if $w_i = 0$ and strictly
convex if $w_i > 0$.

We carry out three algorithms for solving a collection of $50$ random test problems with the size varying from  $M = 10$ to $M=5,000$ components, $m = 5$ to
$300$ coupling constraints and $n = 50$ to $500,000$ variables.
The performance profiles are plotted in Figure \ref{fig:perf_scp} which include the total number of iterations (left) and total computational time (right).
\begin{figure}[ht]
\vskip-0.5cm
\centerline{\includegraphics[angle=0,height=4.0cm,width=12.3cm]{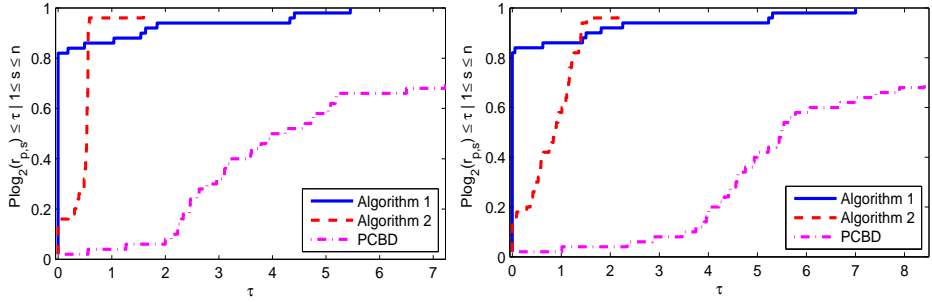}}
\caption{Performance profile of three algorithms in $\log_2$ scale: Left-Number of iterations, Right-CPU time.}
\label{fig:perf_scp}
\vskip-0.5cm
\end{figure}
The numerical test on this collection of problems shows that Algorithm \ref{alg:A1} solves all the problems and Algorithm \ref{alg:A2} solve $48/50$ problems,
i.e. $96\%$ of the collection. \texttt{PCBD} only solves $31/50$ problems, i.e. $62\%$ of the collection.
However, Algorithms \ref{alg:A1} is the most efficient. It solves up to more than $81\%$ problems with the best performance.
\texttt{PCBD} is rather slow and exceeds the maximum number of iterations in many of the test problems ($19$ problems). 
Moreover, it is rather sensitive to the smoothness parameter.

\vskip0.2cm
\noindent\textit{7.2.2. DSL dynamic spectrum management problem.}
In this example, we apply the proposed algorithms to solve a separable convex programming problem arising in DSL dynamic spectrum management.
This problem is a convex relaxation of the original DSL dynamic spectrum management formulation considered in \cite{Tsiaflakis2010}.

Since the formulation given in \cite{Tsiaflakis2010} has an inequality coupling constraint $\sum_{i=1}^Mx_i \leq b$, by adding a new slack variable $x_{M+1}$
such that $\sum_{i=1}^{M+1}x_i = b$ and $0\leq x_{M+1} \leq b$, we can transform this problem into \eqref{eq:separable_convex_problem}.
The objective function of the resulting problem becomes:
\begin{equation}\label{eq:f_i}
\phi_i(x_i) := \begin{cases} 
a_i^Tx_i - \sum_{j=1}^{n_i}c_i^j\ln\left( \sum_{k=1}^{n_i}h_i^{jk}x_i^k + g_i^k \right) &\mathrm{if}~ i=1,\dots, M,\\
0 &\mathrm{if}~ i = M+1.
\end{cases}
\end{equation}
Here, $a_i\in\mathbb{R}^{n_i}$, $c_i, ~g_i\in\mathbb{R}^{n_i}_{+}$ and $H_i := (h_i^{jk})\in\mathbb{R}^{n_i\times n_i}_{+}$, ($i=1,\dots, M$).
The function $\phi_i$ is convex (but not strongly convex) for all $i=1,\dots, M+1$.
As described in \cite{Tsiaflakis2010} that the variable $x_i$ is referred to as transmit power spectral density, $n_i = N$ for all $i=1,\dots, M$ is the
number of users, $M$ is the number of frequency tones which is usually large and $\phi_i$ is a convex approximation of a desired BER
function\footnote{\textbf{B}it \textbf{E}rror \textbf{R}ate function}, the coding gain
and noise margin. A detail model and parameter descriptions of this problem can be found in \cite{Tsiaflakis2010}.

We test three algorithms for the case of $M = 224$ tones and $N = 7$ users. The other parameters are selected as in \cite{Tsiaflakis2010}.
Algorithm \ref{alg:A1} requires $922$ iterations, Algorithm \ref{alg:A2} needs $1314$ iterations, while \texttt{PCBD} reaches the maximum number of
iterations $k_{\max} = 3000$. The relative feasibility gaps $\norm{Ax^k-b}/\norm{b}$ reported by the three algorithms are $9.955\times 10^{-4}$,
$9.998\times10^{-4}$ and $2.431\times10^{-2}$, respectively. 
The obtained approximate solutions of three algorithms and the optimal solution are plotted in Figure \ref{fig:dsl_frequency} which represent the transmit power
with respect to the frequency tones.
\begin{figure}[ht]
\vskip-0.5cm
\centerline{\includegraphics[angle=0,height=7.0cm,width=12cm]{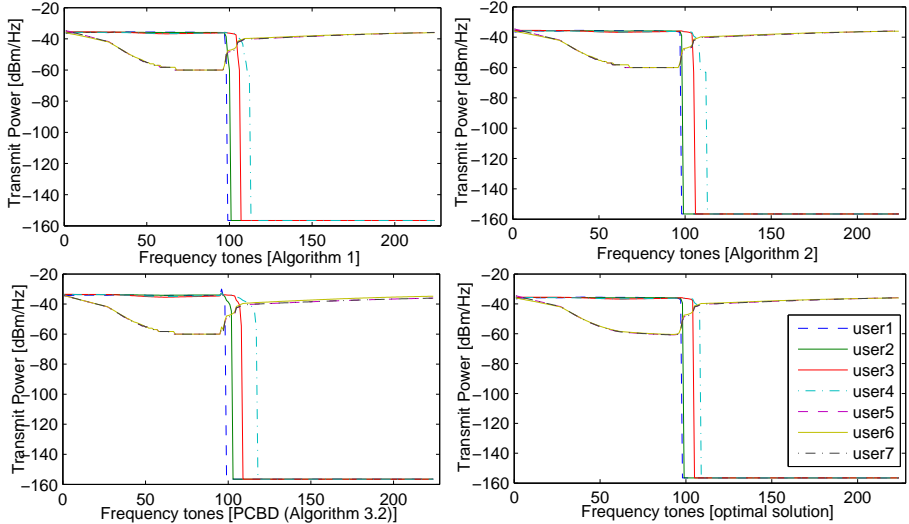}}
\caption{The approximate solutions of the DSL-dynamic spectrum management problem \eqref{eq:scp_prob} reported by three algorithms and the optimal solution.}
\label{fig:dsl_frequency}
\end{figure}
The relative errors of the approximation $x^k$ to the optimal solution $x^{*}$, $\texttt{err}_k := \norm{x^k - x^{*}}/\norm{x^{*}}$, are $0.00853$, $0.00528$
and $0.03264$, respectively. The corresponding objective values are $13264.68530$, $13259.67633$ and $13405.79722$, respectively, while the optimal value is
$13267.11919$.

Figure \ref{fig:dsl_frequency} shows that the solutions reported by three algorithms are consistently close to the optimal one. 
As claimed in \cite{Tsiaflakis2010}, \texttt{PCBD} works much better than subgradient methods. 
However, we can see from this application that Algorithms \ref{alg:A1} and \ref{alg:A2} require fewer iterations than \texttt{PCBD} to reach a relatively
similar approximate solution.

\section{Conclusions}
In this paper, two new algorithms for large scale separable convex optimization have been proposed. Their convergence has been proved and complexity bound has
been given. The main advantage of these algorithms is their ability to dynamically update the smoothness parameters. This allows the algorithms to control the
step-size of the search direction at each iteration. Consequently, they generate a larger step at the first iterations instead of remaining fixed for all
iterations as in the algorithm proposed in \cite{Necoara2008}. The convergence behavior and the performance of these algorithms
have been illustrated through numerical  examples. Although the global convergence rate is still sub-linear, the computational results are remarkable,
especially when the number of variables as well as the number of nodes increase. 
From a theoretical point of view, the algorithms possess a good performance behavior, due to their numerical robustness and reliability. Currently, the
numerical results are still preliminary, however we
believe that the theory presented in this paper is useful and may provide guidance for practitioners. Moreover, the steps of the algorithms are  rather simple
so they can easily be implemented in practice. Future research directions include the dual update scheme and extensions of the algorithms to inexact variants as
well as applications.

\vskip0.3cm
\begin{footnotesize}
\noindent\textbf{Acknowledgments.}
The authors would like to thank Dr. Ion Necoara and Dr. Michel Baes for useful comments on the text and for pointing
out some interesting references. 
Furthermore, the authors are grateful to Dr. Paschalis Tsiaflakis for providing the reality data in the second numerical example. 
Research supported by Research Council KUL: CoE EF/05/006 Optimization in Engineering(OPTEC), IOF-SCORES4CHEM, GOA/10/009 (MaNet), GOA /10/11, several
PhD/postdoc and fellow grants; Flemish Government: FWO: PhD / postdoc grants, projects G.0452.04, G.0499.04, G.0211.05, G.0226.06, G.0321.06, G.0302.07,
G.0320.08, G.0558.08, G.0557.08, G.0588.09, G.0377.09, G.0712.11, research communities (ICCoS, ANMMM, MLDM); IWT: PhD Grants, Belgian Federal Science Policy
Office: IUAP P6/04; EU: ERNSI; FP7-HDMPC, FP7-EMBOCON, ERC-HIGHWIND, Contract Research: AMINAL. Other: Helmholtz-viCERP, COMET-ACCM.
\end{footnotesize}

\appendix
\section*{A. The proofs of Technical Lemmas}
This appendix provides the proofs of two technical lemmas stated in the previous sections.

\vskip0.2cm
\noindent\textbf{A.1. The proof of Lemma \ref{le:intial_point}.} 
The proof of this lemma is very similar to Lemma 3 in \cite{Nesterov2005}.
\begin{proof}
Let $\hat{y} := y^{*}(\hat{x};\beta_2) := \frac{1}{\beta_2}(A\hat{x}-b)$. Then it follows from \eqref{eq:estimate_psi} that:
\begin{eqnarray}\label{eq:lemma24_est1}
\psi(x; \beta_2) \!\! && \overset{\tiny\eqref{eq:estimate_psi}}{\leq} \psi(\hat{x}; \beta_2) + \nabla_1\psi(\hat{x};\beta_2)^T(x_{1}-\hat{x}_{1})  +
\nabla_2\psi(\hat{x};\beta_2)^T(x_{2}-\hat{x}_{2})\nonumber\\
&& + \frac{L_1^{\psi}(\beta_2)}{2}\norm{x_{1}-\hat{x}_{1}}^2 + \frac{L_2^{\psi}(\beta_2)}{2}\norm{x_{2}-\hat{x}_{2}}^2 \nonumber \\
[-1.5ex]\\[-1.5ex]
&& \overset{\tiny\textrm{def.}~\psi(\cdot;\beta_2)}{\!\!\!\!=\!\!\!} \!\!\frac{1}{2\beta_2}\norm{A\hat{x} - b}^2 \!+\! \hat{y}^TA_{1}(x_{1} \!-\! \hat{x}_{1})
\!+\! \hat{y}^TA_{2}(x_{2}  \!-\!
\hat{x}_{2})  \!+\! \frac{L_1^{\psi}(\beta_2)}{2}\norm{x_{1} \!-\! \hat{x}_{1}}^2 \!+\! \frac{L_2^{\psi}(\beta_2)}{2}\norm{x_{2} \!-\!\hat{x}_{2}}^2.\nonumber\\
&& = \hat{y}^T(Ax-b) - \frac{1}{2\beta_2}\norm{A\hat{x}-b}^2 + \frac{L_1^{\psi}(\beta_2)}{2}\norm{x_{1} \!-\! \hat{x}_{1}}^2 \!+\!
\frac{L_2^{\psi}(\beta_2)}{2}\norm{x_{2}-\hat{x}_{2}}^2. \nonumber
\end{eqnarray}
By using the expression $f(x; \beta_2) = \phi(x) + \psi(x; \beta_2)$, the definition of $\bar{x}$, the condition \eqref{eq:initial_point_cond} and
\eqref{eq:lemma24_est1} we have:
\begin{align*}
f(\bar{x}; \beta_2) &\overset{\tiny\eqref{eq:lemma24_est1}}{\leq} \phi(\bar{x}) + \bar{y}^TA_{1}(\bar{x}_{1} - x^c_{1}) + \bar{y}^TA_{2}(\bar{x}_{2} - x_{2}^c) \nonumber\\
& + \frac{L_1^{\psi}(\beta_2)}{2}\norm{\bar{x}_{1} - x_{1}^c}^2 + \frac{L_1^{\psi}(\beta_2)}{2}\norm{\bar{x}_{2} - x_{2}^c}^2 + \frac{1}{2\beta_2}\norm{Ax^c - b}^2\nonumber\\
& \overset{\tiny\eqref{eq:initial_point}}{=} \min_{x\in X}\Big\{ \phi(x) + \frac{1}{\beta_2}\norm{Ax^c - b}^2 + \bar{y}^TA_{1}(x_{1} - x_{1}^c) + \bar{y}^TA_{2}(x_{2} - x_{2}^c) \nonumber\\
& + \frac{L_1^{\psi}(\beta_2)}{2}\norm{x_{1} - x_{1}^c}^2 + \frac{L_2^{\psi}(\beta_2)}{2}\norm{x_{2} - x_{2}^c}^2 \Big\} - \frac{1}{2\beta_2}\norm{Ax^c - b}^2\nonumber\\
& = \min_{x\in X}\left\{ \phi(x) + \bar{y}^T(Ax - b) + \frac{L_1^{\psi}(\beta_2)}{2}\norm{x_{1} - x_{1}^c}^2 + \frac{L_1^{\psi}(\beta_2)}{2}\norm{x_{2} -
x_{2}^c}^2 \right\} - \frac{1}{2\beta_2}\norm{Ax^c - b}^2 \nonumber\\
&\overset{\tiny\eqref{eq:initial_point_cond}}{\leq} \min_{x\in X}\left\{ \phi(x) + \bar{y}^T(Ax - b) + \beta_1[p_1(x_{1}) + p_2(x_{2})] \right\} - \frac{1}{2\beta_2}\norm{Ax^c - b}^2\nonumber\\
& = d(\bar{y}; \beta_1)- \frac{1}{2\beta_2}\norm{Ax^c - b}^2 \leq d(\bar{y}; \beta_1),
\end{align*}
which is indeed the condition \eqref{eq:excessive_gap}.
\Eproof
\end{proof}

\noindent\textbf{A.2. The proof of Lemma \ref{le:update_tau}.}
\begin{proof}
Let us define $\xi(t) := \frac{2}{\sqrt{1+4/t^2}+1}$. It is easy to show that $\xi$ is increasing in $(0,1)$. Moreover, $\tau_{k+1} = \xi(\tau_k)$ for all
$k\geq 0$. Let us introduce $u := 2/t$. Then, we can show that $\frac{2}{u+2} < \xi(\frac{2}{u}) < \frac{2}{u+1}$.
By using this inequalities and the increase of $\xi$ in $(0,1)$, we have:
\begin{equation}\label{eq:lm42_est1}
\frac{\tau_0}{1+2\tau_0k} \equiv \frac{2}{u_0 + 2k} < \tau_k < \frac{2}{u_0 + k} \equiv \frac{2\tau_0}{2+\tau_0k}. 
\end{equation}
Now, by the update rule \eqref{eq:b1b2_update}, at each iteration $k$, we only either update $\beta_1^k$ or $\beta_2^k$. Hence, it implies that:
\begin{eqnarray}
&&\beta_1^k = (1-\tau_0)(1-\tau_2)\cdots(1-\tau_{2\lfloor{k/2\rfloor}})\beta_1^0,\nonumber\\
[-1.5ex]\label{eq:lm42_est2}\\[-1.5ex]
&&\beta_2^k = (1-\tau_1)(1-\tau_3)\cdots(1-\tau_{2\lfloor{k/2\rfloor}-1})\beta_2^0,\nonumber
\end{eqnarray}
where $\lfloor{x\rfloor}$ is the largest integer number which is less than or equal to the positive real number $x$.
On the other hand, since $\tau_{i+1} < \tau_i$ for $i\geq 0$, for any $l\geq 0$, it implies:
\begin{eqnarray}
&(1-\tau_0)\prod_{i=0}^{2l}(1-\tau_i) < \left[(1-\tau_0)(1-\tau_2)\cdots(1-\tau_{2l})\right]^2 <  \prod_{i=0}^{2l+1}(1-\tau_i), \nonumber\\
[-1.5ex]\label{eq:lm42_est3}\\[-1.5ex]
\textrm{and}~ & \prod_{i=0}^{2l-1}(1-\tau_i) < \left[(1-\tau_1)(1-\tau_3)\cdots(1-\tau_{2l-1})\right]^2 <  (1-\tau_0)^{-1}\prod_{i=0}^{2l}(1-\tau_i).
\nonumber
\end{eqnarray}
Note that $\prod_{i=0}^k(1-\tau_i) = \frac{(1-\tau_0)}{\tau_0^2}\tau_k^2$, it follows from \eqref{eq:lm42_est2} and \eqref{eq:lm42_est3} for $k\geq 1$ that:
\begin{eqnarray*}
\frac{(1-\tau_0)\beta_1^0}{\tau_0}\tau_{k+1} < \beta_1^{k+1} < \frac{\beta_1^0\sqrt{1-\tau_0}}{\tau_0}\tau_{k-1},~~\textrm{and}~ \frac{\beta_2^0\sqrt{1-\tau_0}}{\tau_0}\tau_{k+1} <
\beta_2^{k+1} < \frac{\beta_2^0}{\tau_0}\tau_{k-1}. 
\end{eqnarray*}
By combining these inequalities and \eqref{eq:lm42_est1}, and noting that $\tau_0\in (0, 1)$, we obtain \eqref{eq:tau_est2}.
\Eproof
\end{proof}

\bibliographystyle{plain}

\end{document}